\documentclass[11pt]{article}      
\usepackage{amsmath,amsthm,amsfonts,graphicx}
\usepackage{multirow}
\usepackage{multirow}
\def\smallddots{\mathinner{\raise7pt\hbox{.}\raise4pt\hbox{.}\raise1pt\hbox{.}}} 
\def\smallsdots{\mathinner{\raise1pt\hbox{.}\raise4pt\hbox{.}\raise7pt\hbox{.}}}

\DeclareMathOperator{\diag}{diag}
\DeclareMathOperator{\prob}{Prob}
\DeclareMathOperator{\rank}{rank}

\newcommand{\eqid}{\stackrel{d}{=}} 

\newtheorem{theorem}{Theorem}[section] 
\newtheorem{outline}{Outline}[section]
\numberwithin{equation}{section}
\numberwithin{table}{section}
\newtheorem{lemma}{Lemma}[section]

\newtheorem{corollary}{Corollary}[section]

\newtheorem{algorithm}{Algorithm}[section]
\newtheorem{example}{Example}[section]
\newtheorem{definition}{Definition}[section]
\newtheorem{assumption}{Assumption}[section]

\newtheorem{remark}{Remark}[section]

  \usepackage{tikz}
\usetikzlibrary{patterns, patterns.meta}
\usepgflibrary[patterns]

\begin{document} 
 
\title{
CUR Low Rank
Approximation of a Matrix
at Sublinear Cost} 
\author{Victor Y. Pan} 
\author{Soo Go$^{[2],[a]}$,
Qi Luan$^{[2],[b]}$, Victor Y. Pan$^{[1, 2],[c]}$, 
John Svadlenka$^{[2],[d]}$, Liang Zhao$^{[1, 2],[e]}$
\\\\
$^{[1]}$ Department of Computer Science \\
Lehman College of the City University of New York \\
Bronx, NY 10468 USA \\
$^{[2]}$ Ph.D. Programs in  Computer Science and Mathematics \\
The Graduate Center of the City University of New York \\
New York, NY 10016 USA \\
$^{[a]}$ sgo@gradcenter.cuny.edu\\
$^{[b]}$ qi\_luan@yahoo.com \\ 
$^{[c]}$ victor.pan@lehman.cuny.edu \\ 
http://comet.lehman.cuny.edu/vpan/  \\
$^{[d]}$jsvadlenka@gradcenter.cuny.edu \\ 
$^{[e]}$ Liang.Zhao1@lehman.cuny.edu \\
} 
\date{}

\maketitle

  
\begin{abstract}    
Low rank approximation of a matrix (hereafter {\em LRA}) is a highly important area of  Numerical Linear and Multilinear Algebra and Data Mining and Analysis. One can operate  with an LRA  at sublinear cost -- by using much fewer memory cells and flops than an input matrix $M$ has entries.\footnote{``Flop"  stands for ``floating point arithmetic  operation".}  For worst case inputs one cannot
compute even a reasonably close LRA  at sublinear cost, but in computational practice  accurate LRAs, even in their memory efficient form of CUR LRAs, are routinely obtained  at sublinear cost for large and important classes of matrices, in particular by means of Cross-Approximation iterations, which specialize
Alternating Direction techniques to LRA. We identify some classes of matrices
for which CUR LRA are computed at sublinear cost as well as some   sublinear cost LRA algorithms that are empirically accurate for 
large classes of inputs. Some of our techniques
and concepts can be of independent interests.
\end{abstract}
 




\bigskip
\medskip


  

\section{Introduction}\label{sintr} 


{\bf 1.1. LRA at sublinear cost:  background and a challenge.}
LRA of a matrix is among the most fundamental problems of 
Numerical Linear and Multilinear Algebra and Data  Mining and Analysis, with  
  applications ranging from machine
learning theory and neural networks
to term document data and DNA SNP data (see surveys \cite{HMT11,M11,KS17}). 
For example, matrices that represent 
  Big Data (e.g., unfolding matrices of multidimensional tensors)
tend to be so immense that  only a tiny fraction of the entries
fits primary memory of a computer, but  
   quite typically  they admit LRA \cite{UT19},\footnote{Here and throughout we use such concepts as ``low", ``small", ``nearby", etc.  defined in context.} 
that is, lie close to  low rank matrices,
with which one can operate 
at  sublinear cost. 

Given an $m\times n$
 matrix $M$ and a {\em target rank} \footnote{ ``In practice, the target rank is rarely known in advance. $\dots$ LRA algorithms are
usually implemented in an adaptive fashion $\dots$
until the error norm satisfies the desired tolerance"
\cite[Sec. 2.4]{TYUC17}.} $\rho$, a very simple algorithm  of Sec. \ref{scyn} (we call it 
{\em Primitive})
computes rank-$\rho$ 
approximation\footnote{This means ``approximation having  
rank at most $\rho$".} of $M$ by a matrix product $CUR$ 
where  $R$ and $C$ denote two submatrices made up
of $k$ rows and $l$
columns of $M$, respectively, for a fixed pair of integers $k$ and $l$ such that
$\rho\le \min\{k,l\}$.
The algorithm uses
{\em sublinear memory space}, that is, much fewer
than $mn$ scalars,
if $kl\ll mn$ and runs at sublinear cost  if $kl\min\{k,l\}\ll mn$ (see (\ref{eqcnncur})). 
 
CUR preserves sparsity and non-negativity of  an input matrix; it is a memory  efficient LRA, widely applied in data analysis.
 Its output  accuracy  crucially depends on the choice of 
  the row and columns sets or equivalently  the $k\times l$ submatrix $G$ shared by the matrices $R$ and $C$ and said to be {\em CUR generator}.
Zamarashkin and Osinsky proved in   
\cite{ZO18}
that for any $m\times n$ matrix $M$, any positive integer $k\le\min\{m,n\}$, and some
$k\times k$ CUR generator  
$G$, the  Primitive algorithm  optimizes, up to a factor of $k+1$, approximation 
of $M$  under the Frobenius matrix norm,\footnote{By extending \cite{ZO18}, Cortinovis and
 Kressner \cite{CK20}   computed such a CUR LRA by using $O(kmn^3)$ ops for $m\ge n$.}
whereas  the error norms of CUR LRAs of $M$ are unbounded 
 over ill-conditioned generators $G$
  and more generally  any reasonably close LRA {\em fails
miserably}
on 
the worst case inputs and even on the small matrix families of Example \ref{exdlt} unless
all $mn$ entries of $M$ are involved.

Nevertheless, computation of accurate LRAs at sublinear cost, in particular by means of
{\em Cross-Approximation (C-A)} iterations, which adjust  the ADI  celebrated method \cite{S16}
 to LRA 
(see  Secs. 1.4 and \ref
{scurlra}),
routinely succeeds empirically for a large and important  class of
matrices. This raises 
 {\bf two Challenges} that we meet here.
 \medskip
 
{\bf 1.} Characterize matrices
for which 
the Primitive algorithm outputs  meaningful  CUR LRAs.
 
{\bf 2.} Specify some LRA algorithms   
that run at sublinear cost or at least use sublinear memory space
 and are empirically  accurate 
 for large classes of matrices. 
 \medskip

\noindent  {\bf 1.2. Accurate CUR LRA for  large input  classes.} In Thm. \ref{thm:error_estimate}
we estimate the {\em spectral  norm}  $||E||$ of the output error matrix $E=M-CUR$
 of the Primitive algorithm   
for rank-$\rho$, namely,
we deduced that

\begin{equation}\label{eqerrcur}
||M-CUR||<(3.3\frac{v+1}{1-\theta} +2)(v+1)\epsilon
 \end{equation}
 for the CUR  of rank $\rho\le\min\{k,l\}$ output by the Primitive algorithm provided that 
 
\begin{equation}\label{eqtheps1}
v:=||U||\max\{||C||,||R||\},~\theta:=\epsilon ||U|| < 1,
\end{equation}

\begin{equation}\label{eqeps1}
\min_
{\rank(M')\le \rho}||M'-M||\le \epsilon.
\end{equation}

For $\theta< 1/2$, say,  upper bound
(\ref{eqerrcur}) on $||E||$
exceeds its 
optimal value 
by quite a large factor
 of $(6.6(v+1)+2)(v+1)$ but can be interesting qualitatively -- it implies that the Primitive algorithm 
 does not fail where the generators $G$ are well-conditioned, while it  tends to fail   
otherwise.

We  strengthen    bound (\ref{eqerrcur}) 
where
 $M$  has a gap between its $\rho$th and $(\rho+1)$st largest singular values (see bound (\ref{eqgap})) 
  and is a small norm perturbation  of a  {\em factor-Gaussian matrix} 
 \begin{equation}\label{eqn:fgdfn}
    M':= H_1 \Sigma H_2,
\end{equation}
for  a $\rho\times \rho$
 well-conditioned matrix\footnote{Without loss of generality, {\em wlog}
  (see Remark \ref{repr3}),  we can assume that $\Sigma$ is a diagonal matrix,  banded with rows or columns filled with 0s where $m\neq n$.} $\Sigma$  of small full rank $\rho$ and  two independent   {\em Gaussian random} matrices
  $H_1$ and $H_2$, 
  filled with independent identically distributed    
 normal  (standard Gaussian) random variables. 
In that case we deduce that, under the assumptions of 
Thm. \ref{thm:error_estimate} 
for $\theta=1/2$, the CUR approximation
output by the Primitive algorithm has a spectral error norm 
\begin{equation}\label{eqerrnrm}
||E||:=||M-CUR||=O(\epsilon\cdot \max\{\frac{m}{k},\frac{n}{l}\})
\end{equation}
with a probability at least 0.6.

The latter bound only  applies
to a  
 narrow class of matrices $M$ because of the gap assumption  of (\ref{eqgap}), 
but the concept of factor-Gaussian matrices
can be of independent interest.
 We obtain (\ref{eqn:fgdfn})
by replacing   the  factors of SVD of $M'$ filled with left and right singular vectors by Gaussian random matrices  $H_1$ and  $H_2$, respectively;  
 $H_1$ and  $H_2$ are skinny 
and hence 
 close to matrices with
orthonormal columns \cite{RV09}
for  $\rho\ll \min\{m,n\}$. 
 By virtue of  Thm. \ref{thquasi},
{\em  Gaussian pre-processing}, that is,  pre- and/or post-multiplication by Gaussian random matrices, turns  any rank-$\rho$ matrix
into a  
 factor-Gaussian matrix of                                                                                                                                                                                                                                                                                                                                                                                                                                                                                                                                                                                                                                                                                                                                                                                                                                                                                                                                                                                                                                                                                                                                                                                                                                                                                                                                                                                                                                                                                                                                                                                                                                                                                                                                                                                                                                                                                                                                                                                                                                                                                                                                                                                                                                                                                                                                                                  expected rank $\rho$ and  hence turns  any matrix
that admits close rank-$\rho$
approximation into a small norm perturbation of such a factor-Gaussian matrix.
At the very end of Sec.  \ref{spreprmlt} 
we point out a link of Gaussian pre-processing to
random sketching LRA.
$~$

\noindent {\bf 1.3.
Acceleration of near-optimal LRA
algorithms.}
(i) Various random sketching LRA algorithms 
 are {\em  expected 
to output near-optimal LRAs}, that is,  output them with a high probability
{\em (whp)} 
(see \cite{HMT11,
M11,TYUC17,
MT20,
TWa}, and the bibliography therein). These algorithms   
run at sublinear cost apart from their stage of Gaussian pre-processing,
which runs at superlinear cost. 

That stage runs at sublinear cost as well where
Gaussian
multipliers are replaced with proper Ultrasparse 
multipliers
 such as Abridged 
SRHT matrices of Appendix \ref{ssrht}.
Then the output
 LRAs are computed at sublinear cost, are not expected to be near-optimal anymore, and
 empirically can be poor approximations \cite{L09}, 
 but   tend to be quite accurate 
for  a large   class of matrices $M$ and proper choice of
Ultrasparse
multipliers, according to formal and   empirical
study in
\cite{PLSZb}.

(ii) The random sampling algorithms of \cite{DMM08}
compute CUR LRA expected to be near-optimal 
under the Frobenius norm. 
Applied 
 to skinny matrices involved into C-A iterations  for LRA, these algorithms run  at 
 sublinear cost.  So do also the  entire hybrid  CUR  algorithms, combining C-A iterations with random sampling of \cite{DMM08}, provided that the iterations converge fast,
which is quite usually observed
empirically.
 In the tests  for various  real
world matrices
(see our Tables \ref{tabcadmm1} and \ref{tabcadmm}, reproduced from  
\cite{PLSZa}),
the output accuracy of the resulting LRA was consistently within  small factors from that of \cite{DMM08},
expected to be near-optimal. 
 \smallskip
   
\noindent {\bf 1.4. Related work.} 
Our study was largely motivated by  empirical  efficiency of C-A  (alternating direction) iterations for computing accurate CUR LRA at sublinear cost (see  \cite{T96,GZT95,GZT97,
GTZ97,T00,B00,GT01,BR03,
BG06,GOSTZ10,B11,
GT11,OZ18,O18,ALS24}, and the bibliography therein).
This  has led us                                                       to the two
Challenges of Sec. 1.1, responded in the papers   \cite{PLSZ16,
PLSZ17,PLSZa,
PLSZb,PLSZ20,
LP20,PLa,
LPSa}, which cite \cite{PQY15,PZ17a,
PZ17b} as their technical predecessors. Our current work elaborates upon the unpublished
results of \cite{PLSZa}, which extended 
\cite{PLSZ16,
PLSZ17}. 
 
Sublinear cost algorithms, called superfast,  have been studied extensively for
Toeplitz, Hankel, Vandermonde, Cauchy, and other structured matrices having small displacement rank  and defined by small number of parameters
(see \cite{P01,
XXG12,P15}
and
extensive bibliography therein).
More recently, randomized LRA algorithms
running at sublinear cost
have been proposed in
\cite{MW17,
BW18,
CETW23}  for some special but large and important classes of matrices
defined by large numbers of independent parameters. Most notably, the authors of  \cite{MW17,
CETW23} proved that their algorithms are expected
to output near-optimal LRAs for Symmetric Positive
Semidefinite (SPSD) matrices.\footnote{The sublinear cost of the   LRA algorithms of \cite{MW17,
BW18,CETW23}  does not include the
 superlinear cost of a posteriori  estimation of  their output error norms and correctness
verification, but the  deterministic algorithm of \cite[Part III]{LP20}, running at sublinear cost, computes LRA of an
$n\times n$ SPSD matrix
with both spectral and Frobenius error norms within a factor of $n$ from optimal and as by-product,  at no additional cost, estimates
error norm, verifying correctness.}

\cite[Alg. 2.2]{LYMHY} is a superfast heuristic  
CUR LRA algorithm, which is accurate for a large class of  unsymmetric
matrices
according to the test results in \cite{LYMHY}.  The algorithm
combines uniform random choice of some subsets of the column and row sets for CUR with their updating
by means of the Interpolative 
Decomposition  of \cite[Sec. 3.2.3]{HMT11}
  based on 
Strong Rank Revealing (SRR) QR  factorization of  \cite{GE96}.

Jianlin Xia in \cite{X24}
calls this
technique
``progressive
alternating direction pivoting", combines it with some additional
techniques, in particular,
randomized
error estimation,  and   consistently obtains highly  accurate  LRAs for  real world inputs in his extensive numerical
tests. 

In spite of impressive results of numerical tests,
 the heuristic
sublinear cost algorithms
of \cite{LYMHY,X24}
as well as
randomized sublinear cost estimation of  their accuracy   
in \cite{X24} must fail, e.g., for
the matrices
  of our Example \ref{exdlt}.

Xia cites
\cite{LP20,
PLSZ20}
as his close predecessors,
but his algorithms, 
 as well as \cite[Alg. 2.2]{LYMHY}, are quite different from our current ones -- e.g., he only uses randomization for error estimates, and  his randomization techniques  have no overlap with  ours.
 
In \cite{CD13}
 Chiu and Demanet achieve important 
 progress in formal support of the accuracy of LRA algorithms running at sublinear cost. They
prove that under the
 uniform random choice of the sets
 of  row and column indexes defining a
 $q\times q$ CUR generator,  
 the resulting rank-$r$ approximation of $M$
 is expected to be quite accurate provided that  
$M\approx XYZ^T$, for  
$X\in \mathbb R^{m\times r}$,
$Y\in \mathbb R^{r\times r}$,
$Z^T\in \mathbb R^{r\times n}$, $q$ a little exceeding $r$,
 and  the matrices $X$ and $Z$ having orthonormal columns and  incoherent, that is, filled with entries of
 comparable magnitude
(see \cite[Def. 1.1]{CD13} for formal definition). 
If only $X$ has   orthonormal columns and is incoherent,
then  Chiu and Demanet
define  CUR by performing a single C-A step based on \cite{GE96}  and then still prove 
similar  accuracy estimate  for that CUR LRA.

For a limitation,  their basic provision
of incoherence
 does not hold for a large class of inputs and  cannot  be verified at sublinear cost.

By applying  novel advanced
techniques,  Cortinovis and Ying  \cite{CY25} deduced the results of \cite{CD13}
under  weaker assumptions
about the factor matrices
$X$ and $Z$ -- by allowing some of the column vectors to be sparse rather than incoherent;
they proved  this for the algorithm of \cite{CD13} that incorporates progressive
alternating direction pivoting  of  \cite{LYMHY,X24}  
 into the algorithm of \cite{CD13}. The admirable progress in \cite{CY25} still shares the cited limitation of  \cite{CD13}.

 \cite{CD13,
CY25} characterize the classes of matrices    
 whose
LRA can be computed at sublinear cost but, unlike us,  do this in terms of
 incoherence
 and sparsity,
and their techniques are
different  from ours.
 
 \smallskip
 
\noindent {\bf 1.5.
 Organization of our paper.} 
   We  recall some  background material
   in the next four sections. 
    In Secs. \ref{scgrbckg}
 and \ref{serrrnd} we estimate  output errors of canonical CUR LRA of general matrices and perturbed random
    matrices, respectively. 
  In Sec.  \ref{spreprmlt} we prove that  Gaussian pre-processing  turns  any rank-$\rho$ matrix
into   
a  factor-Gaussian matrix of expected rank $\rho$.
 In Sec. \ref{sexpr} we cover our numerical experiments.
 In Appendix  \ref{shrdin}
    we specify some 
   small families of matrices whose LRA fails
unless all input entries are involved.
In Appendix
\ref{ssrht}  we define 
Abridged SRHT  Ultrasparse
 matrices. 
 In Appendix \ref{svlmfct} we estimate the
volume of a factor-Gaussian matrix.

  \section{Some background for LRA}\label{sbckgr}
  
    $\mathbb R^{p\times q}$  denotes the class of $p\times q$ 
  real  matrices. 
For simplicity  we assume  dealing
with real matrices 
throughout,\footnote{Hence the Hermitian transpose $M^*$
is just the transpose $M^T$.} but our study can be extended to complex matrices; in particular 
 see \cite{E88,CD05,ES05,TYUC17} for some relevant results about complex Gaussian matrices.

  In this section our notation  $|\cdot|$ unifies  the spectral norm  $||\cdot||$ and the Frobenius norm
  $||\cdot||_F$.
 
 Given a   tolerance $\epsilon$ to an output error norm, 
an $m\times n$ matrix $M$
 has $\epsilon$-rank at most $\rho$
 if it admits approximation within  $\epsilon$
 by a matrix $M'$ of rank at most $\rho$
 or equivalently if  there exist three matrices $A$, $B$, and $E$ such that
\begin{equation}\label{eqlra}
M=M'+E~{\rm where}~|E|\le \epsilon 
|M|,~M'=AB,~A\in \mathbb R^{m\times \rho},~{\rm and}~B
\in \mathbb R^{\rho\times n}.
\end{equation}
$M'=AB$ is a two-factor  LRA if $\rho$ is small in context. 

\bigskip

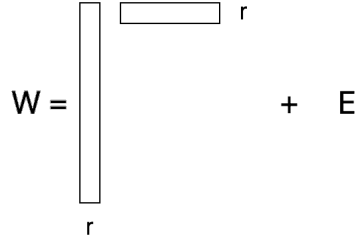
\begin{figure}[ht] 
\centering
\begin{tikzpicture}[scale=1]

\def\gap{0.3}

\node at (-1.5, 1.5) {$M$};

\node at (-0.5, 1.5) {$=$};

\draw (0, 0) rectangle (0.3, 3);
\node at (0.15, -.4) {$\rho$};

\draw (0.3+1*\gap, 2.7) rectangle (2.2+1*\gap, 3);
\node at (2.3+2*\gap, 2.8) {$\rho$};

\node at (3.2+2*\gap,1.5) {$+$};

\node at (4.2+2*\gap,1.5) {$E$};

\end{tikzpicture}

\caption{Rank-$\rho$ approximation of a matrix $M$}\label{fig1}
\end{figure}  

  A 2-factor LRA $AB$ of $M$ of (\ref{eqlra})   can be generalized to a 3-factor LRA: 
\begin{equation}\label{eqrnkrho}
 M=M'+E,~|E|\le \epsilon|M|,~M'=ATB,~A\in \mathbb R^{m\times k},~T\in \mathbb R^{k\times l},~B\in \mathbb R^{l\times n},
\end{equation}  
\begin{equation}\label{eqklmn}
\rho\le k\le m,~\rho\le l\le n
\end{equation}
 for $\rho=\rank(M')$, and typically $k\ll m$ and/or $l\ll n$.
Each of the two pairs of maps
 $$\{AT\rightarrow A;~B\rightarrow B\}~{\rm and}~\{A\rightarrow A;~TB\rightarrow B\}$$ turns a 3-factor LRA $ATB$ of (\ref{eqrnkrho})
into a 2-factor LRA $AB$ of                                                                                                                            (\ref{eqlra}).
  
An important 3-factor LRA of $M$ is 
its $\rho$-{\em top
SVD} (see
Fig. \ref{fig5}),
$M_{\rho}=S_{\rho}\Sigma_{\rho} T_{\rho}^*$
 for a diagonal matrix  $\Sigma_{\rho}=
 \diag(\sigma_j)_{j=1}^{\rho}$ of the $\rho$ largest singular values of $M$ and
  two  matrices $S_{\rho}$ and $T_{\rho}$ of the $\rho$ associated left and right singular vectors, 
  respectively.
  $M_{\rho}$ is said to be
  the $\rho$-{\em truncation} of $M$. 
  

\begin{figure}
[ht]
\centering 
\begin{tikzpicture}[scale=1]

\def\gap{0.3}

\draw (-3.5,0) rectangle (-1, 3);
\node at (-2.2,1.5) {$M$};

\node at (-0.5,1.5) {$=$};

\draw (0,0) rectangle (0.5,3);

\draw (0.5+\gap,2.5) rectangle (1+\gap,3);
\draw (0.5+\gap,3) -- (1+\gap,2.5);
\node at (0.65+\gap,2.66) {\tiny 0};
\node at (0.87+\gap,2.86) {\tiny 0};

\draw (1+2*\gap,2.5) rectangle (3.5+2*\gap,3);

\node at (4+2*\gap,1.5) {$+$};

\draw (4.5+2*\gap,0) rectangle (7+2*\gap, 3);
\node at (5.75+2*\gap,1.5) {$E$};

\end{tikzpicture}
\caption{The figure 
represents   
top SVD of a matrix as well as its  CUR LRA.}  
\label{fig5}
\end{figure}  
     
\begin{theorem} \label{thtrnc} (Eckart-Young-Mirsky Theorem, see {\rm \cite[Thm. 2.4.8]{GL13}.)} 
It holds that $$\tau_{\rho+1}(M):=
\min_{N:~\rank(N)=\rho} |M-N|=|M-M_{\rho}|$$ 
under both spectral and Frobenius norms:
$\tau_{\rho+1}(M)=\sigma_{\rho+1}(M)$ 
under the spectral norm, and 
$\tau_{\rho+1}(M)=\sigma_{F,\rho+1}(M):=\sqrt{\sum_{j> \rho}\sigma_j^2(M)}$ 
under the Frobenius norm.
\end{theorem}


 \begin{theorem}\label{thsngr} {\rm \cite[Cor. 8.6.2]{GL13}.}
 For a pair of ${m\times n}$ matrices $M$ and $M+E$ it holds that
 $$|\sigma_j(M+E)-\sigma_j(M)|\le||E||~{\rm for}~j=1,\dots,\min\{m,n\}. $$
  \end{theorem} 
  
Hereafter $M^+$ denotes the Moore--Penrose pseudo inverse of $M$.

  
 \begin{lemma}\label{lehg} {\rm (The norm of the pseudo inverse of a matrix product, see, e.g., \cite{GTZ97}.)}
Suppose that $A\in\mathbb R^{k\times r}$, 
$B\in\mathbb R^{r\times l}$
and the matrices $A$ and 
$B$ have full rank 
$r\le \min\{k,l\}$.
Then
$$|(AB)^+|
\le |A^+|~|B^+|.$$ 
\end{lemma}
  
\section{CUR decomposition
and CUR LRA}\label{scurlra}

 
    
 For two sets $\mathcal I\subseteq\{1,\dots,m\}$  
and $\mathcal J\subseteq\{1,\dots,n\}$  -- define
the submatrices
$$M_{\mathcal I,:}:=(m_{i,j})_{i\in \mathcal I; j=1,\dots, n},  
M_{:,\mathcal J}:=(m_{i,j})_{i=1,\dots, m;j\in \mathcal J},~{\rm and}~ 
M_{\mathcal I,\mathcal J}:=(m_{i,j})_{i\in \mathcal I;j\in \mathcal J}.$$
   
Given an $m\times n$ matrix $M$ of rank 
$\rho$
and its  nonsingular 
$\rho\times \rho$ submatrix
$G=M_{\mathcal I,\mathcal J}$ one can readily verify that 
$M=M'$ for
\begin{equation}\label{eqcurd} 
M'=CUR,~C=M_{:,\mathcal J},~
  U=G^{-1},~G=M_{\mathcal I,\mathcal J},~{\rm and}~R=M_{\mathcal I,:}.  
\end{equation}   
We call $G$ the {\em generator} and   
 $U$ the {\em nucleus} of {\em CUR decomposition} of $M$ (see Fig. \ref{fig2}).

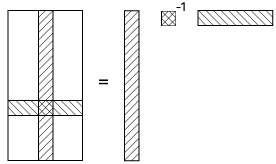
\begin{figure}
[ht]
\centering
\begin{tikzpicture}

\draw (0,0) rectangle (2,4); 

\draw (0.8,0) rectangle (1.2,4);
\fill[pattern=north east lines] (0.8,0) rectangle (1.2,4);

\draw (0,1.1) rectangle (2,1.5);
\fill[pattern=north west  lines] (0,1.1) rectangle (2,1.5);

\node at (2.7,2) {$=$};

\draw (3.3,0) rectangle (3.7,4);
\fill[pattern=north east lines] (3.3,0) rectangle (3.7,4);

\draw (4.1,3.6) rectangle (4.5,4);
\fill[pattern=crosshatch] (4.1,3.6) rectangle (4.5,4);
\node at (4.6,4.15) {\scriptsize$-1$};

\draw (4.9,3.6) rectangle (6.8,4);
\fill[pattern=north west lines] (4.9,3.6) rectangle (6.8,4);

\end{tikzpicture}
\caption{CUR decomposition  with a nonsingular CUR generator}
\label{fig2}
\end{figure}

CUR decomposition is extended to {\em CUR approximation}  of
a matrix $M$ close to a rank-$\rho$ 
matrix (see Fig. \ref{fig5}),  although the approximation
$M'\approx M$
 for $M'$ of (\ref{eqcurd}) tends to  be poor where the generator $G$ is 
 ill-conditioned.
 \begin{remark}\label{recgr} 
The pioneering  
 papers  \cite{GZT95,GTZ97,GZT97},  as well as \cite{GT01,GT11,GOSTZ10,OZ18}, define CGR  approximations having 
 nuclei  $G$; ``G" can
 stand, say, for
``germ". We use the acronym CUR, which is more customary in the West. 
 ``U" can stand, say, for ``unification factor", but notice the alternatives of CNR, CCR, or CSR with    
$N$, $C$, and $S$ standing for {\em ``nucleus",
``core", and ``seed"}.
\end{remark}

 
 

By generalizing  (\ref{eqcurd}) we allow to use
$k\times l$ CUR generators
for $k$ and $l$ satisfying (\ref{eqklmn})
and to choose any $l\times k$ nucleus
$U$ for which the error matrix
$E=CUR-M$ has smaller norm.


Given two matrices $C$ and $R$,   
 the  minimal Frobenius 
 error norm  of CUR LRA 
$$||E||_F=||M-CUR||_F\le ||M-CC^+M||_F+||M-MR^+R||_F$$ is reached   for the nucleus
$U=C^+MR^+$
(see \cite[Eqn. (6)]{MD09}).  
We, however, cannot compute such a nucleus  at sublinear cost and instead seek   
{\em canonical CUR LRA} (cf. \cite{DMM08,CLO16,OZ18}) whose
 nucleus is the Moore-Penrose pseudo inverse of the $\rho$-truncation of a given CUR generator:
\begin{equation}\label{eqcnncur}
U:=G_{\rho}^+. 
\end{equation}
Given a generator $G$ we can compute $G_{\rho}^+$  by using about $kl$
scalars and $O(kl\min\{k,l\})$ flops.

  
  \begin{theorem}\label{thncl} {\rm [A necessary and sufficient criterion for CUR decomposition.]}
   Let $M'=CUR$ be a  canonical CUR of $M$ for
 $U=G_{\rho}^+$, $G=M_{\mathcal I,\mathcal J}$. Then $M'=M$ if and only if $\rank(G) = \rank(M)$. 
\end{theorem}
\begin{proof}  $\sigma_j(G)\le \sigma_j(M)$ for all $j$
 because $G$ is a submatrix of $M$. Hence 
  $\epsilon$-rank$(G)\le \epsilon$-rank$(M)$ for all nonnegative $\epsilon$,
   and  in particular $\rank(G) \le \rank(M)$.                                                                  
 
 Now let $M=M'=CUR$. Then clearly $$\rank(M)\le \rank(U)=\rank(G_{\rho}^+)=\rank(G_{\rho})\le \rank(G),$$
 Hence  $$\rank(G) \ge \rank(M),~{\rm and~so}~ 
 \rank (G)=\rank (M)~{\rm if}~M'=M.$$
 
It remains to deduce that $$M=CG_{\rho}^+R~{\rm if}~\rank(G) =\rank(M):=\rho,$$
but in this case  $G_{\rho}=G$, and so 
$$\rank(CG_{\rho}^+R)=\rank(C)=\rank(R)=\rho.$$
Hence the rank-$\rho$ matrices $M$ and $CG_{\rho}^+R$ share  their rank-$\rho$ submatrices 
$C$
and $R$.
\end{proof}

 \begin{remark}\label{rencl}
 Can we extend the theorem by proving that  $M'\approx M$
 if and only if   $\epsilon$-rank$(G) = \epsilon$-rank$(M)$ for a small positive $\epsilon$? The ``only if'' claim  cannot be extended, e.g., for
 $$
 M = 
 \begin{pmatrix}
 1 & 0 & 0\\
 0 & \epsilon & 0\\
 0 & 1 & 0
 \end{pmatrix},$$
$\epsilon\approx 0$, and the $2\times 2$ leading submatrix  $G$ of $M$.
Indeed, $\rank(M)=
 \rank(G)=2$, and so Thm. \ref{thncl} implies
 that $M'=M$, while
 $\epsilon$-$\rank(M)=2>\epsilon$-$\rank(G)=1.$
 \end{remark}

  
\section{CUR LRA algorithms running at sublinear cost}\label{sprimc}
  

\subsection{Primitive and Cynical algorithms}\label{scyn}
  
        
 Given an $m\times n$
matrix $M$ admitting close rank-$\rho$ approximation
and a pair of $k$ and $l$ satisfying (\ref{eqklmn}), 
 define a canonical CUR LRA of $M$
  for a 
fixed or chosen at random pair  
 of  sets
 $\mathcal I$ and $\mathcal J$ of $k$ row and $l$ column indexes, respectively, and  
 call the resulting algorithm 
 {\bf Primitive}. Apart from the selection 
 of the sets 
 $\mathcal I$ and $\mathcal J$ the algorithm  computes the $\rho$-truncation 
 $(M_{\mathcal I,\mathcal J})_{\rho}$
 of the matrix 
 $M_{\mathcal I,\mathcal J}$ and its pseudo inverse 
$((M_{\mathcal I,\mathcal J})_{\rho})^+$
by using about $kl$  scalars and  $O(kl\min\{k,l\})$ flops. 

The following CUR LRA algorithm
  (we call it   {\bf Cynical})\footnote{Here we allude to  the benefits of the austerity and simplicity of primitive life,   
advocated by Diogenes the Cynic, and not to shamelessness and  distrust associated with modern  cynicism.}  first fixes a $p\times q$
submatrix of $M$ and then  computes its $k\times l$ submatrix  by applying a compression  algorithm of \cite{GE96,
P00,DMM08}, using about $pq$ scalars and  
$O(pq\min\{p,q\})$ 
 flops.

  \begin{figure}[ht] 
\centering
\begin{tikzpicture}

\draw (0,0) rectangle (4,3);

\fill[pattern=north east lines] 
    (1.2,1.25) rectangle (2,2.05);
\draw (1.2,1.25) rectangle (2,2.05);

\def\bsize{0.3}
\fill[black] 
    (1.4,1.45) rectangle  (1.65,1.7);

\end{tikzpicture}
\caption{A cynical CUR algorithm (the stripes mark a $p\times q$ submatrix; a $k\times l$ CUR generator is shown in black).
}\label{fig3} 
\end{figure}
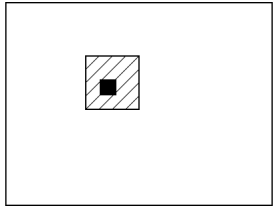

\begin{algorithm}\label{algcnc} 
For  an $m\times n$  matrix $M$,  a target rank $\rho$, and four  integers $k$, $l$, $p$, $q$ such that   
\begin{equation}\label{eqklmnpqr}  
0<\rho\le k\le p\le m,~\rho\le l\le q\le n,~{\rm and}~ kl<pq,
\end{equation}
fix or  randomly sample  
a pair  
 of  sets
 $\mathcal I$ and $\mathcal J$ of $p$ row and $q$ column indexes, respectively,  compute a $k\times l$ CUR generator
 $G_{k,l}$
for the  $p\times q$ submatrix
$M_{\mathcal I,\mathcal J}$ of $M$
by applying to $M_{\mathcal I,\mathcal J}$ one of the algorithms of 
\cite{GE96,
P00,DMM08}, compute the $\rho$-truncation $G_{k,l,\rho}$, and 
  build on it   CUR LRA of $M$.
\end{algorithm}
  

The two-stage choice of the CUR generator
in a cynical algorithm  decreases the error norm;
application of the algorithms of \cite{GE96,
P00} can increase  the output error bound of Alg. \ref{algcnc} by a factor of
$\sqrt{(p-k)(q-l)kl}$, expected to disappear
where compression uses randomization of \cite{DMM08}.



\subsection{Cross-Approximation (C--A) iterations}\label{scait}

  {\bf C--A iterations} (see Fig. \ref{fig4}):

 \begin{itemize} 
\item
For  an $m\times n$ matrix $M$ and 
target rank $r$, fix four integers $k$, $l$, $p$ and $q$ satisfying   
(\ref{eqklmnpqr}). [C-A iterations are simplified in a special case where 
 $p:=k$ and $q:=l$.]  
\item
Fix an $m\times q$ ``vertical" submatrix of the matrix $M$, made up of its $q$ columns.\footnote{One can alternatively begin C--A iterations with a ``horizontal" submatrix.} 
\item
By applying a fixed CUR LRA subalgorithm, e.g., one of the algorithms of \cite{O18,
GE96,
P00,
DMM08},\footnote{Such a subalgorithm runs at sublinear cost on skinny inputs  involved in C--A iterations,
although the  algorithms of 
\cite{O18,GE96,P00,DMM08}
 run at superlinear cost on  an $m\times n$ matrix $M$.}
 compute 
a $k\times l$ CUR generator $G$
of this submatrix\footnote{At this stage one can apply 
progressive
alternating direction pivoting of \cite{LYMHY,X24,
CY25}.} and reuse it for the matrix $M$. 

\item
Output the resulting CUR LRA of $M$ if it is close enough. 
\item
Otherwise swap $p$ and $q$ and reapply
the algorithm to the  matrix $M^T$.

[This is equivalent to computing a $k\times l$ CUR generator of a fixed  
 $p\times n$ ``horizontal" submatrix $M_1$ of $M$ that covers the submatrix $G$.]
 \item
 Recursively alternate such ``vertical" and  
  ``horizontal" steps until the new  CUR generator
  computed at the current iteration  coincides with the original one or until the number of recursive C--A steps exceeds a fixed tolerance bound. 
 \end{itemize}

      

 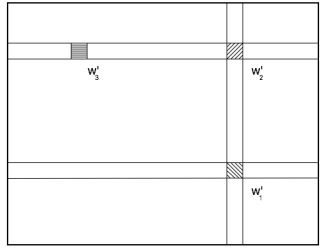
\begin{figure} 
[ht]  
\centering
\begin{tikzpicture}
\def\cols{20}
\def\rows{16}
\def\cellsize{0.25}
\def\width{\rows*\cellsize}

\draw (14*\cellsize, 0) -- (14*\cellsize, \rows*\cellsize);
\draw (15*\cellsize, 0) -- (15*\cellsize, \rows*\cellsize);

\draw (0, \width - 3*\cellsize) -- (\cols*\cellsize, \width - 3*\cellsize);
\draw (0, \width - 4*\cellsize) -- (\cols*\cellsize, \width - 4*\cellsize);
\draw (0, \width - 11*\cellsize) -- (\cols*\cellsize, \width - 11*\cellsize);
\draw (0, \width - 12*\cellsize) -- (\cols*\cellsize, \width - 12*\cellsize);

\draw[thick] (0,0) rectangle (\cols*\cellsize, \rows*\cellsize);
\draw (4*\cellsize, \width - 4*\cellsize) rectangle (5*\cellsize, \width - 3*\cellsize);

\draw[pattern={Lines[angle=-45,distance=1.25pt]}]
    (14*\cellsize, \width - 4*\cellsize) rectangle (15*\cellsize, \width - 3*\cellsize);

\draw[pattern={Lines[angle=45,distance=1.25pt]}]
    (14*\cellsize, \width - 12*\cellsize) rectangle (15*\cellsize, \width - 11*\cellsize);

\draw[pattern={Lines[angle=90,distance=1.25pt]}]
    (4*\cellsize, \width - 4*\cellsize) rectangle (5*\cellsize, \width - 3*\cellsize);
\end{tikzpicture}
\caption{The first three recursive C--A steps output three
striped submatrices.}
\label{fig4}
\end{figure}
 
{\bf Initialization recipes.} For a large class of input matrices
 it is  efficient to apply   
  {\em adaptive C-A iterations} (cf. \cite{B00,B11,BG06,BR03}). For initialization they adapt
 Gaussian elimination with  pivoting combined with dynamic
 search for gaps in the spectrum of the singular values of $M$. The  alternative initialization
in \cite{GOSTZ10} uses  $O(n\rho^2)$ flops to initialize the C--A iterations for an $n\times n$ input and
 $q=s=\rho$, and then  the algorithm uses $O(\rho n)$ flops per C--A step.

\subsection{Volume maximization and {\em CUR LRA}}
\label{svmmxcur}
 
 
 The error norm  of CUR LRA is
 estimated at linear or superlinear cost for worst case inputs (whp if randomization is applied), but empirically  one should stop
 C-A iterations
 as soon as 
 new CUR generator coincides with the previous one.

  \cite{GOSTZ10,OZ18,O18}   achieves this  in finite (and empirically small) number of steps by increasing the 
 volume or projective volume of CUR generator 
at every C-A step. According to 
Thm. \ref{th3} of \cite{OZ18}, recalled below,
 the output CUR  closely approximates $M$ under rather mild
 conditions. We first recall the relevant definitions.
 
 \begin{definition}\label{defmxv}
For   
three integers $k$,  $l$, and $r$ such that
$1\le r\le \min\{k,l\}$,
define the {\em volume} $v_2(M):=\prod_{j=1}^{\min\{k,l\}} \sigma_j(M)$ 
of a  $k\times l$ matrix $M$ and its $r$-{\em projective volume}
$v_{2,r}(M):=\prod_{j=1}^r \sigma_j(M)$
 such that
$v_{2,r}(M)=v_{2}(M)~{\rm if}~r=\min\{k,l\}$,
$v_2^2(M)=\det(MM^*)$ if $k\ge l$;
$v_2^2(M)=\det(M^*M)$ if $k\le l$,
$v_2^2(M)=|\det(M)|^2$ if $k = l$.
\end{definition}

 Given a matrix $W$, five integers $k$,  $l$, $m$, $n$, and $r$ such that $1\le r\le \min\{k,l\}$, and a real $h>1$, and an $m\times n$ 
matrix $W$, 
 its $k\times l$ submatrix $G$ has  {\em locally $h$-maximal} volume (resp. $r$-projective volume)  in $W$ if  $v_2(G)$ (resp. $v_{2,r}(G)$) is maximal up to a factor of $h$ among all $k\times l$                                                                                                                                                                                                                                                                                                                                                   submatrices
 of $W$ that differ from $G$ in a single row and/or a single column.
 We	write {\em maximal} for $1$-maximal
 and drop ``locally" if the volume maximization is over all submatrices of
 a fixed size $k\times l$. 
 
For an $m\times n$ matrix $W=(w_{i,j})_{i,j=1}^{m,n}|$ 
  define its Chebyshev norm $||W||_C:=\max_{i,j=1}^{m,n}|w_{ij}|$  such that
  (cf. \cite{GL13}) 
\begin{equation}\label{eq0}
||W||_C\le ||W||_2 \le ||W||_F\le \sqrt {mn}~||W||_C.
\end{equation}

 Recall the following result of  \cite{OZ18}
 extending \cite{GZT95,
 GTZ97,GZT97,GT01}.

\begin{theorem}\label{th3}  
Suppose that  
 $W_{k,l}=W_{\mathcal I,\mathcal J}$
is a $k\times l$ submatrix of 
 an $m\times n$ matrix $W$,
 $U=W_{k,l,r}^+$ 
 is the canonical nucleus of
a  {\em CUR LRA} of $W$, $E=W-CUR$, $h\ge 1$,
 and the $r$-projective volume of
$W_{\mathcal I,\mathcal J}$ is  locally  $h$-maximal.  
Then 
$$||E||_C\le h~f(k,l,r)~\sigma_{r+1}(W)
~~{\rm for}~~f(k,l,r):=
\sqrt{\frac{(k+1)(l+1)}{(k-r+1)(l-r+1)}}. 
$$
\end{theorem}

  \begin{remark}\label{re3}  
For $r=\min\{k,l\}$, the $r$-projective volume turns into volume;  then $f(k,l,r)$ turns into  $\sqrt{\frac{(k+1)(l+1)}{(|l-k|+1)}}$ and \cite{ALS24}
   strengthens 
the bound on the norm $||E||_C$ a little.
\end{remark}

\section{Background for random matrix computations}\label{ssdef}

 
 
\subsection{Gaussian and factor-Gaussian matrices of low rank}

Hereafter $\mathbb E(w)$ denotes the expected value of random variable $w$,  $\eqid$ 
denotes the equality  in distribution,      
 $\preceq$ and $\succeq$ denote {\it statistically less or equal to} and
{\it statistically greater or equal to}, respectively, and $\mathcal{G}^{m\times n}$ denotes
an $m\times n$ random Gaussian matrix. 



\begin{theorem}\label{thrnd} {\em [Nondegeneration of a Gaussian Matrix.]}
Let $F\eqid\mathcal G^{r\times m}$, 
$H\eqid\mathcal G^{n\times r}$, $M\in  
\mathbb R^{m\times n}$, and $r\le\rank(M)$. 
Then 
the matrices $F$, $H$, $FM$,  and $MH$  
have full rank $r$ 
with probability 1.
\end{theorem}

 
\begin{proof}
Fix any of the matrices 
 $F$, $H$,  $FM$, and $MH$ and its $r\times r$ submatrix $B$. Then 
 the equation $\det(B)=0$ 
  defines an algebraic variety of a lower
dimension in the linear space of the entries of the matrix because in this case $\det(B)$
is a polynomial of degree $r$ in the entries of the matrix $F$ or $H$
(cf. \cite[Proposition 1]{BV88}). 
Clearly, such a variety has Lebesgue  and
Gaussian measures 0, both being absolutely continuous
with respect to one another. This implies  the theorem.
\end{proof}


\begin{assumption}\label{assmp1}  {\em [Nondegeneration of a Gaussian matrix.]}
Throughout this paper we simplify the statements of our results by assuming that a Gaussian matrix  has full rank and ignoring the chance   for its degeneration, which has probability 0. 
\end{assumption}

In this section and  Sec. \ref{spreprmlt}
we call an $m\times n$ (possibly rectangular) matrix $M$ {\em orthogonal}
 if it has orthonormal rows and/or columns  or equivalently 
if $M^*M$ and/or  $MM^*$ is an identity matrix.


\begin{lemma}\label{lepr3} {\rm [Orthogonal invariance of a Gaussian matrix.]}
Suppose that $k$, $m$, and $n$  are three  positive integers, $k\le \min\{m,n\},$
$G\eqid \mathcal G^{m\times n}$, $S\in\mathbb R^{k\times m}$, 
 $T\in\mathbb R^{n\times k}$, and
$S$ and $T$ are orthogonal matrices.
Then $SG$ and $GT$ are Gaussian matrices.
\end{lemma} 


\begin{definition}\label{deffctrg1}  {\em [Factor-Gaussian matrices.]} 
Suppose that $\rho \le \min\{m, n\}$, 
$$H_1 \eqid \mathcal{G}^{m\times \rho}~{\rm and}~H_2 \eqid \mathcal{G}^{\rho \times n}$$ are 
two independent random Gaussian matrices, and  $$A\in \mathbb{R}^{m\times \rho},~B\in\mathbb{R}^{\rho\times n},~{\rm and}~\Sigma\in\mathbb{R}^{\rho\times\rho}$$ are full rank {\it well-conditioned} 
constant 
matrices. 

(i) Then we call $AH_2$, $H_1B$, and $H_1\Sigma H_2$ right, left, and two-sided factor-Gaussian matrices of rank $\rho$, respectively
(cf. Assumption \ref{assmp1}). 

(ii) We refer to small-norm perturbations of 
factor-Gaussian matrices  of rank $\rho$ as to {\em perturbed} 
right, left, and two-sided
factor-Gaussian matrices of rank $\rho$ as well as to
right, left, and two-sided
factor-Gaussian matrices of $\epsilon$-{\em rank} $\rho$ for a fixed small positive $\epsilon$ (see
Fig. \ref{fig5}).
\end{definition} 


\begin{remark}\label{repr3}
Substitute SVD $\Sigma=SDT$
into the product $H_1\Sigma H_2$
and rewrite it as
$H'_1DH'_2$
where $H'_1=H_1S$
and $H'_2=TH_2$
are Gaussian random matrices by   virtue of Lemma \ref{lepr3}.
Hence in our definition of a two-sided Gaussian matrix we can  assume wlog that $\Sigma$
is a diagonal matrix, possibly banded with rows or columns filled with 0s.
\end{remark}

   
  
 
 

\subsection{Norms of a Gaussian  matrix and its pseudo inverse
}\label{snrmg}
Hereafter  we only use the spectral norm $||\cdot||$ of matrices, $\Gamma(x)=
\int_0^{\infty}\exp(-t)t^{x-1}dt$
denotes the Gamma function,   and we write 



\begin{equation}\label{eqexpct}
\mathbb E||M||:=\mathbb E(||M||)~{\rm and}~e:=2.7182818\dots.
\end{equation}

\begin{definition}\label{defnrm}{\rm [Norms of a Gaussian matrix and its pseudo inverse.]} 
Define random variables $\nu_{m, n} \eqid ||\mathcal{G}^{m\times n}||$ and $\nu^+_{m, n} \eqid ||(\mathcal{G}^{m\times n})^+||$.
\end{definition}


\begin{theorem}\label{thsignorm} {\rm [Expected norm of a Gaussian matrix.]}

  
 
 $\mathbb E(\nu_{m,n})\le \sqrt m+\sqrt n$.
\end{theorem}
\begin{proof}
See \cite[Thm. II.7]{DS01}.
\end{proof}


\begin{theorem}\label{thsiguna} 
 {\rm [Expected norm of the inverse of a Gaussian matrix.]} 


  
 $\mathbb E(\nu^+_{m,n})\le \frac{e\sqrt{m}}{m-n}$ 
provided that $m\ge n+2\ge 4$.

\end{theorem}


\begin{proof}
 See 
%
\cite[Eqn. (10.4)]{HMT11}.
\end{proof}
 

\begin{remark}\label{resig}
\cite[Thm. II.7]{DS01}, 
\cite[the proof of Lemma 4.1]{CD05}, \cite[Prop. 10.4 and Eqn. (10.3)]{HMT11}, and \cite[Thm. 3.3]{SST06} estimate 
probability distribution  
of $\nu_{m,n}$ 
and $\nu^+_{m,n}$, but in this paper 
we already succeed by combining Thms.  \ref{thsignorm} and \ref{thsiguna} with  Markov inequality.
 \end{remark}



\section{A posteriori errors of a canonical CUR LRA}\label{scgrbckg} 


\subsection{ Error Estimation: an Outline and the Statement} 
 We estimate the error norm of CUR rank-$\rho$ approximation of $M$ by comparing it with CUR
 decomposition of a nearby rank-$\rho$ matrix $M'$ defined by the same sets of row and column indices.

\begin{outline}\label{pr1} {\rm [Error Estimation for  a Canonical CUR LRA.]} 
   \begin{enumerate}
  \item
  Consider (but do not compute) an auxiliary 
  $m\times n$ matrix $M'$ of rank $\rho$ that approximates the  matrix  $M$ within a fixed norm bound $\epsilon$ such that 
 \begin{equation}\label{eqmm'} 
\sigma_{\rho+1}(M)\le  ||M-M'||\le  \epsilon.
\end{equation} 
  [We can apply our study to any choice of $M'$. E.g.,
  $||M-M'||=\epsilon:=\sigma_{\rho+1}(M)$
  for a natural choice of                                                                                                                                                                                                                                                                                                                                                                                                                                                                                                                                                                                                                                                                                                                                                                                                                                                                                                                                                                                                                                                                                                                                                                                                                                                                                                                                                                                                                                                  $M'=M_{\rho}$, but 
 in Sec. \ref{serrrnd} we apply this study to matrix $M$ being 
  a norm-$\epsilon$ perturbation of a factor-Gaussian matrix $M'$ of  Def. \ref{deffctrg1}.]  
  \item
For the matrices $M$ and  $M'$ fix  two 
 row and column 
index sets
$\mathcal I$ and $\mathcal J$, respectively,
and define  $k\times l$
generators  $G=M_{\mathcal I,\mathcal J}$ and $G'=M_{\mathcal I,\mathcal J}'$,   nuclei $U=G_{\rho}^+$ and $U'=G_{\rho}'^+$,
and canonical  CUR approximation 
 $M\approx CUR$ and  decomposition $M'=C'U'R'$. 
  \item
  Observe that
\begin{equation}\label{eqwc} 
|| M- C U R||\le
|| M-M'||+||M'- C U R||\le \epsilon + ||C'U'R'- C U R||.
\end{equation}
\item%
Bound the norm
$||C'U'R'- C U R||$ in terms of the values
$\epsilon$, $||C||$, $||U||$, and $||R||$.
\end{enumerate}
\end{outline}  

Next we  
elaborate upon step 4 provided 
that we have already performed steps 1 -- 3. 

\begin{theorem}\label{thm:error_estimate}
Given an $m\times n$ matrix $M$, a $k\times l$ matrix $G := M_{\mathcal{I}, \mathcal{J}}$, a positive integer $\rho < \min\{k, l\}$,  and an $m\times n$ rank-$\rho$ matrix $M'$
satisfying (\ref{eqmm'}) for $\epsilon\le \sigma_{\rho}(G)$, 
write 
\begin{equation}\label{eqv}
C:=M_{:,\mathcal{J}},~R:=M_{\mathcal{I}, :},~ 
 U = G_{\rho}^+,
\end{equation}
\begin{equation}\label{eqzt}
\zeta := 
\begin{cases}
\sqrt{2} & {\rm for} ~~\rho = \min\{k, l\} \\
(1+\sqrt{5})/2 &{\rm for} ~~\rho < \min\{k, l\}
\end{cases}
\end{equation}
and assume that $v$ and $\theta$ satisfy 
(\ref{eqtheps1}).
 Then it  holds that
 \begin{equation}\label{eqn:error_estimate}
    || M - CUR || \le (v+1)\Big(\frac{2\zeta}{1-\theta} ~(v+1)+ 2\Big)\epsilon,~{\rm where}~2\zeta\le 
1+\sqrt{5}~{\rm (cf.~ (\ref{eqzt}))}.
 \end{equation}
\end{theorem}

 In view of (\ref{eqwc}) we only need to estimate the norm $||C'U'R' - CUR||$.

\subsection{The first bound on
$||C'U'R' - CUR||$}
\begin{lemma}\label{thwc}  
Fix  five integers  
$k$, $l$, $m$, $n$, and $\rho$ such that $\rho\le k\le m$
and $\rho\le l\le n$,  an  $m\times n$ matrix $M$,  
its rank-$\rho $ approximation $M'$  satisfying
(\ref{eqmm'}), 
and canonical CUR LRAs
$$ M\approx C U R~{\rm and}~M'=C'U'R'$$
defined by the same pair of 
 index sets
$\mathcal I$ and  $\mathcal J$ of cardinality $k$ and $l$, respectively, such that 
$$C:=M_{:,\mathcal J},~
R:=M_{\mathcal I,:},~U=G_{\rho }^+,~
~ C':= M'_{:,\mathcal J},~R':= M'_{\mathcal I,:},~U'=G_{\rho}'^+,$$
$$ G=M_{\mathcal I,\mathcal J},~{\rm and}~G'=M'_{\mathcal I,\mathcal J}.$$  
 Then   
$$||C' U' R'- C U R||\le 
(||R||+||C'||)~|| U||~\epsilon+  
||C'|| ~||R'||~||U'- U||.$$
\end{lemma}
\begin{proof}
Notice that   
$$ C U R-C'U'R'=( C-C')  U R+
C'  U( R-R')+ 
C'( U-U')R'.$$
Therefore
\begin{align*}
||  C U R-C'U'R'|| &\le
|| C-C'||~|| U||~|| R||+  
||C'||~|| U||~|| R-R'||+
||C'||~|| U-U'||~||R'||.
\end{align*}
Substitute the bound
$\max\{||C-C'||,||R-R'||\}\le ||M-M'||\le \epsilon$.
\end{proof}

\subsection{Estimation of the norm  $||U-U'||$}\label{spstr1}

 
  

\begin{lemma}\label{leprtpsdinv}   
Under the assumptions of 
Lemma \ref{thwc}
 we have
$$
||U - U'|| \le 2\zeta~||U||~||U'||~\epsilon~
  {\rm for}~\zeta~{\rm of}~(\ref{eqzt}).$$ 
\end{lemma}

\begin{proof} 
Recall that $\rank(G') = \rank(M') = \rank(G_{\rho}) = \rho$ and that $$||G' - G_{\rho}|| \le ||G' - G|| + ||G - G_{\rho}|| \le ||M' - M|| + \sigma_{\rho + 1}(M) \le 2\epsilon.$$ Then apply \cite[Thm. 2.2.5]{B15} for $A=G,~B=G'$. 
\end{proof}



\begin{lemma}\label{leprtinv}
Under the assumptions of Lemma \ref{leprtpsdinv}
let (\ref{eqtheps1}) hold. Then
 $||U'||\le \frac{||U||}{1-\theta}$.
\end{lemma}

\begin{proof}
Thm. \ref{thsngr} implies that  
$$\sigma_{\rho}(G') \ge \sigma_{\rho}(G)-\epsilon=(1-\epsilon/\sigma_{\rho}(G))\sigma_{\rho}(G)=(1-\epsilon~||U||)\sigma_{\rho}(G).$$   
Substitute (\ref{eqtheps1}) and obtain that 
$\sigma_{\rho}(G') \ge
(1-\theta)\sigma_{\rho}(G)>0$.
 Hence 
$$||U'||=\frac {1}{\sigma_{\rho}(G')}\le \frac{1}{1-\theta}\cdot \frac {1}{\sigma_{\rho}(G)} = \frac{||U||}{1 - \theta}.$$  
\end{proof}

\begin{corollary}
\label{couu'}
Under the assumptions of Lemma \ref{leprtpsdinv}
it holds that
$
||U - U'|| \le \frac{2\zeta}{1-\theta}~||U||^2~\epsilon$.
\end{corollary}

\subsection{Proof of Thm. \ref{thm:error_estimate}}

Combine bound (\ref{eqwc}) and  Lemma \ref{thwc}  and 
obtain
\begin{equation}\label{eqmcur}
||M-C U R||\le \epsilon + (||R||+||C'||)||U||\epsilon+||C'||~||R'||~||U-U'||.
\end{equation}

Recall (\ref{eqmm'}) and obtain
$||C'||\le ||C|| + \epsilon$. 

Combine  (\ref{eqtheps1})
and Cor.    
\ref{couu'}  and deduce that $$||C'||~||U-U'||\le (||C||+\epsilon) \frac{2\zeta}{1-\theta}~||U||^2\epsilon<\frac{2\zeta}{1-\theta}(||C||~|U||^2+||U||)\epsilon.$$ 
Hence
$$||C'||~||R'||~||U-U'||<\frac{2\zeta}{1-\theta}(||R'||~||C||~||U||^2+||R'||~||U||)\epsilon.$$
Substitute  relationships  (\ref{eqv}) and  obtain
\begin{equation}\label{eqcru}
||C'||~||R'||~||U-U'||<
\frac{2\zeta}{1-\theta}(v+1)^2
\epsilon.
\end{equation}

Likewise, deduce from the bounds $||C'||\le ||C|| + \epsilon$ and 
 $\epsilon||U||<1$ that 
$$(||R||+||C'||)||U||\le 
(||R||+||C||+\epsilon)||U||< (||R||+||C||)||U||+1.$$

Recall the bound on $v$ from
(\ref{eqv})
and obtain that
$(||R||+||C'||)||U||<2v+1$.

 Combine this bound with Eqns. (\ref{eqmcur})  and (\ref{eqcru})
 and  obtain the theorem.



  
\section{The errors of CUR LRA of a perturbed factor-Gaussian matrix}\label{serrrnd}

Next  assume that
a matrix $M$ is close to  a two-sided factor-Gaussian 
 matrix $M'$ of low rank
$\rho$  and then strengthen the
estimate
of Thm. \ref{thm:error_estimate}
 by proving  bound (\ref{eqerrnrm}).  
  
  
Begin with auxiliary results.
Consider a matrix $M'$ 
of (\ref{eqn:fgdfn})
for $H_1$, $H_2$, and $\Sigma$ 
of Def. \ref{deffctrg1} and
for  $m,n,k,l$, and $\rho$ satisfying
\begin{equation}\label{eqn:fgpreq}
    k < m,~l < n,~\textrm{and}~ \min\{k, l\}\ge \rho + 2 \ge 4. 
\end{equation}
As in (\ref{eqexpct})
write $e:=2.7182818\dots$ and define parameters 
$\alpha$ and $\beta$ 
as follows:
\begin{equation}\label{eqn:alpbeta}
    \alpha := \frac{e^2\sqrt{kl}}{(k-\rho)(l-\rho)}
    ~\textrm{and}~ 
    \beta := \max\big\{ (\sqrt{m} + \sqrt{\rho})(\sqrt{\rho} + \sqrt{l}), (\sqrt{n} + \sqrt{\rho})(\sqrt{\rho} + \sqrt{k}) \big\}.
\end{equation}


 

\begin{theorem}\label{thfgcur}  
 For two Gaussian matrices $H_1$ and $H_2$,
 matrices $\Sigma$ and $M'$ 
of (\ref{eqn:fgdfn}),    five integers $k, l, m, n, \rho$  satisfying  (\ref{eqn:fgpreq}), and 
parameters $\alpha$ and $\beta$ of (\ref{eqn:alpbeta}), let $\mathcal I$ and $\mathcal J$ be
two fixed
 row and column 
index sets, respectively,  such that
$$C'=M'_{:,\mathcal J},~ 
R'=M'_{\mathcal I,:},~
G'= M'_{\mathcal I,\mathcal J}\in \mathbb R^{k\times l}, ~{\rm and}~U'=G'^+.$$ 
Then $C'$,  $R'$, and $G'$ are factor-Gaussian matrices such that
\begin{eqnarray*}
    \mathbb{E} ||C'|| \le (\sqrt{m} + \sqrt{\rho})(\sqrt{\rho} + \sqrt{l})\sigma_1(\Sigma)\le \beta\sigma_1(\Sigma),~~~~~~~~~ \\
    \mathbb{E} ||R'|| \le (\sqrt{n} + \sqrt{\rho})(\sqrt{\rho} + \sqrt{k})\sigma_1(\Sigma)\le \beta\sigma_1(\Sigma),~~~~~~~~~\\
    \mathbb{E}||U'|| 
   \le \frac{e^2\sqrt{kl}}{(k-\rho)(l-\rho)\sigma_\rho(\Sigma)}\le \frac{\alpha}{\sigma_\rho( 
    \Sigma)}  ~~~~~~~~~~~~~~~~~~~~~
    \end{eqnarray*}
~{\rm for}~ $e:=2.7182818\dots$~{\rm of ~(\ref{eqexpct}).}
\end{theorem}
\begin{proof}
 Since $C' = M'_{:, \mathcal{J}} =  H_1\Sigma{H_2}_{:, \mathcal{J}}$ and
 $H_1$ and $H_2$ are independent
 of one another,
 obtain
\begin{align*}
    \mathbb{E} ||C'|| &\le \mathbb{E} \big( ||H_1||~||\Sigma||~||{H_2}_{:,\mathcal{J}}||\big)\\
    &= \sigma_1(\Sigma) \mathbb{E}||H_1|| ~ \mathbb{E}||{H_2}_{:,\mathcal{J}}||.
\end{align*}
Recall from Thm. \ref{thsignorm} that $\mathbb{E}(\nu_{p, q}) \le \sqrt{p} + \sqrt{q}$
and obtain
\begin{equation*}
    \mathbb{E} ||C'|| \le (\sqrt{m} + \sqrt{\rho})(\sqrt{\rho} + \sqrt{l})\sigma_1(\Sigma).
\end{equation*}
 One can 
similarly  estimate
$\mathbb{E}||R'||$.

Next apply Lemma \ref{lehg},
recall from Thm. \ref{thsiguna} that 
$$\mathbb{E}(\nu^+_{p, q}) \le \frac{e\sqrt{p}}{p - q},~{\rm for}~p \ge q+2 \ge 2,$$  and deduce that 
\begin{align*}
\mathbb{E}||U'|| &= \mathbb{E}||({H_1}_{\mathcal{I}, :}\Sigma{H_2}_{:, \mathcal{J}})^+||\\
&\le \mathbb{E}||{H_1}_{\mathcal{I}, :}^+||~\mathbb{E}||{H_2}_{:, \mathcal{J}}^+||~/\sigma_{\rho}(\Sigma)   \\
&\le \frac{e^2\sqrt{kl}}{(k-\rho)(l-\rho)\sigma_\rho(\Sigma)}= \frac{\alpha}{\sigma_\rho(\Sigma)}.
\end{align*}

\end{proof}
 
 
Now  
fix two  
 sets $\mathcal{I}$ and $\mathcal{J}$ 
of its row and column indexes
of our perturbed two-sided factor-Gaussian matrix $M$ and then
deduce  
that whp  the generator $G = M_{\mathcal{I}, \mathcal{J}}$ is well-conditioned  and bound
the norms  $||C||$ and $||R||$. 
 
\begin{lemma}\label{lemma:factor_gaussian_prob}Under the assumptions of Thm. \ref{thfgcur}, let
\begin{equation}\label{eqgap}
 \sigma_{\rho+1}(M)\le \epsilon \le \frac{\sigma_\rho(\Sigma)}{30\alpha},
 \end{equation} 
let $||E|| \le \epsilon$ and 
write $$M:= M' + E,~C := M_{:, \mathcal{J}},~R:= M_{\mathcal{I}, :},~{\rm and}~U: = (M_{\mathcal{I}, \mathcal{J}})_\rho^+.$$ Then with a probability no less than 0.7 we have 
$$ {\rm (i)}~~~~~~~~~~~~~~~~~~~~~~~~~~~~~~~~~~~~~~~~~||U|| \le \frac{15\alpha}{\sigma_{\rho}(\Sigma)},~~ \epsilon ||U|| \le 1/2, ~~~~~~~~~~~~~~~~~~~~~~~~~~~~~~~$$

$$ {\rm (ii)}~~~~~~~~~~~~~~~~~~~~~~~~~~~
v:=\max\{ ||C||, ||R|| \}||U|| \le 150\alpha\beta\sigma_1(\Sigma)/\sigma_\rho(\Sigma) + 1/2.~~~~~~~~~~~~~~~~~~$$
\end{lemma}

\begin{proof}
(i) By combining the bound of Thm. \ref{thfgcur} on the expected norm
$\mathbb  E||U'||$ with Markov's inequality,  deduce that 
\begin{equation}\label{equ'alph}
 ||U'|| \le \frac{10\alpha}{\sigma_\rho(\Sigma)} \textrm{ or equivalently } \sigma_{\rho}(M'_{\mathcal{I}, \mathcal{J}}) \ge \frac{\sigma_\rho(\Sigma)}{10\alpha}
\end{equation}
with 
a probability no less than 0.9.

Similarly to the argument of Lemma \ref{leprtinv},  deduce that $$\sigma_{\rho}(M_{\mathcal{I}, \mathcal{J}}) \ge \sigma_{\rho}(M'_{\mathcal{I}, \mathcal{J}}) - \epsilon \ge \frac{\sigma_\rho(\Sigma)}{15\alpha},$$ and then claim (i) 
follows.\\
(ii) By combining the  bounds of Thm. \ref{thfgcur} on the expected norms  
$\mathbb||C'||$ and $\mathbb||R'||$ with
 Markov's inequality and the  bound on  the
 probability of the union of random variables, 
 obtain that 
$$
\max\big\{||C||,||R||\big\} \le 10\beta\sigma_{1}(\Sigma) + \epsilon
$$
with 
  a
 probability no less than 0.8.
Then combine claim (i) with the union bound
and obtain that
\begin{equation*}
    \max\{ ||C||, ||R|| \}||U|| \le 150\alpha\beta\sigma_1(\Sigma)/\sigma_\rho(\Sigma) + 1/2  
\end{equation*}
 with 
a probability no less than 0.7.
\end{proof}

Now fix 
 any generator of CUR LRA of $M$ and estimate 
 $||M-CUR||$. 

\begin{theorem}\label{thm:pfg_error_bound}
If the 
assumptions
of Lemma \ref{lemma:factor_gaussian_prob}
hold for  a perturbed factor-Gaussian matrix $M$, then with a probability no less than 0.6
bound (\ref{eqn:error_estimate})  on 
the approximation error norm $||M - CUR||$ for $\theta\le 1/2$ holds, that is, 
 \begin{equation}\label{eqn:factor_gauss_error_est} 
  ||M - CUR|| \le (     4\zeta (v+3)(v+1) \epsilon,~{\rm for}~4\zeta\le 2+2\sqrt 5,~v = 150\alpha\beta\sigma_1(\Sigma)/\sigma_\rho(\Sigma) + 1/2.
\end{equation}
\end{theorem}

\begin{proof}
Combine  bounds (\ref{eqgap}) 
 and (\ref{equ'alph})  
and conclude that
with a probability $p_0\ge 0.9$, the bounds
 $\sigma_\rho(G)\ge 3\epsilon > \epsilon$
  and hence 
(\ref{eqtheps1})
  hold. Therefore, we can apply Thm. \ref{thm:error_estimate} and obtain
(\ref{eqn:error_estimate}). 
 Replace $\theta$ and $v$ with their upper bounds 1/2 and $ 150\alpha\beta\sigma_1(\Sigma)/\sigma_\rho(\Sigma) + 1/2$,  respectively, which hold with a 
 probability $p_1\ge 0.7$.
 Combine the bounds $p_0\le 0.9$ and $p_1\le 0.7$.
\end{proof}

\begin{remark}\label{remark:factor_gaussian}
Recall from (\ref{eqn:alpbeta}) that 
$$\alpha := \frac{e^2\sqrt{kl}}{(k-\rho)(l-\rho)}~{\rm and}~\beta := \max\big\{ (\sqrt{m} + \sqrt{\rho})(\sqrt{\rho} + \sqrt{l}), (\sqrt{n} + \sqrt{\rho})(\sqrt{\rho} + \sqrt{k}) \big\},$$ make simplifying assumptions that
$\sigma_1(\Sigma)/\sigma_\rho(\Sigma) = O(1),$
 $m \gg k \gg \rho$, $n \gg l \gg \rho$, so that $k-\rho\approx k$ and 
$l-\rho\approx  l$,
 drop the smaller terms of bounds (\ref{eqn:alpbeta})  
  and (\ref{eqn:factor_gauss_error_est}), and obtain  dominant parts   of the estimates for
  $\alpha$, $\beta$, $v$, and $||M-CUR||$
  as follows,
$\alpha\approx \frac{e^2}{\sqrt{kl}}$,   
$\beta\approx
\max\{\sqrt{ml},\sqrt{nk}\}$,
$v=O(\max\{\sqrt{m/k},\sqrt{n/l}\})$, and
(\ref{eqerrnrm})
 holds, that is, $ ||M-CUR||= O\Big(\max\Big\{\frac{m}{k},\frac{n}{l} \Big\}\cdot\epsilon \Big).$
\end{remark}

\section{Gaussian pre-processing
and Generalized Nystr{\"o}m algorithm}\label{spreprmlt}


 
Next we prove that pre-processing with 
Gaussian multipliers $X$ and $Y$ transforms any matrix that admits LRA
into a perturbation of a factor-Gaussian matrix. 

\begin{theorem}\label{thquasi}  
Consider five integers $k$, $l$, $m$, $n$, and  $\rho$
satisfying (\ref{eqklmn}),
an $m\times n$ well-conditioned matrix $M$ 
of rank $\rho$, $k\times m$ and $n\times l$ Gaussian matrices $G$ and $H$, respectively,
and the norms $\nu_{p,q}$
and $\nu_{p,q}^+$ of Def. \ref{defnrm}.
Then 

(i) $GM$ is a left factor-Gaussian matrix of  rank $\rho$ such that
$$||GM||\preceq ||M||~\nu_{k,\rho}~{\rm  
and}~||(GM)^+||\preceq ||M^+||~\nu_{k,\rho}^+,$$

(ii) $MH$ is a right factor-Gaussian matrix of  rank $\rho$ such that
$$||MH|| \preceq ||M||~\nu_{\rho,l}~{\rm 
and}~ 
||(MH)^+||\preceq ||M^+||~\nu_{\rho,l}^+,$$ 
(iii) $GMH$ is a two-sided 
factor-Gaussian matrix of rank 
$\rho$
such that $$||GMH||\preceq ||M||~\nu_{k,\rho}\nu_{\rho,l}~{\rm 
and}~||(GMH)^+||\preceq ||M^+||~\nu_{k,\rho}^+\nu_{\rho,l}^+.$$
\end{theorem}
\begin{proof}
Let $M=S_M\Sigma_MT^*_M$ be SVD
where $\Sigma_M$ is  the diagonal matrix of the singular values of $M$;
it is well-conditioned since so is 
the matrix $M$.
Then $\bar G:=GS_M$ and $\bar H:=T_M^*H$
are Gaussian matrices by virtue of  
 Lemma \ref{lepr3}, which states 
orthogonal invariance of Gaussian matrices: indeed in our case    
$G$ and $H$ are Gaussian, 
  while
$S_M\in \mathbb R^{m\times \rho}$ and $T_M\in \mathbb R^{\rho\times n}$ are orthogonal  matrices. Furthermore,                                                                                                                                                                                                                                                                                                                                                                                                                                                                                                
  
(i) $GM=\bar G\Sigma_MT_M^*= 
\bar G_{\rho}\Sigma_MT_M^*$,

(ii) $MH=S_M\Sigma_M\bar H=
S_M\Sigma_M\bar H_{\rho}$, 

(iii) $GMH=\bar G\Sigma_M\bar H=
\bar G_{\rho}\Sigma_M\bar H_{\rho}$
where  $\rho\le \min\{m,n\}$,    and

(iv) $\bar G_{\rho}=
\bar G\begin{pmatrix}I_{\rho}\\O\end{pmatrix}$,
and 
$\bar H_{\rho}=(I_{\rho}~|~O)\bar H$.

 Combine the latter claims (i)--(iv) with Lemma \ref{lehg}.                                                       \end{proof}

Now successively compute the matrices $GM$, $MH$, and $GMH$  above,  $N:=(GMH_{\rho})^+$, and $\widehat  M:=MHNGM$,
and observe that  this computation is precisely
the numerically stable Generalized Nystr{\"o}m algorithm of \cite{N20}; hence  
$\widehat 
M$ is expected to be a near-optimal rank-$\rho$ approximation of $M$.  


\section{Numerical experiments 
}\label{sexpr}


\subsection{Test overview}

 
Next we present our tests of 
Primitive, Cynical, and C--A algorithms for CUR LRA of both synthetic and real world input matrices. Our tests have confirmed  high efficiency of C-A iterations,  performed at sublinear cost, in accordance to and even beyond 
our formal support. Moreover, their power was impressively strong even without using adaptive techniques of \cite{B00,B11}. The tests even 
showed some potentials of computing accurate at sublinear cost the LRAs
of matrices of large classes 
 by means of Primitive
or Cynical algorithms.

We have performed
 the tests   
 in the Graduate Center of the City University of New York 
by using  MATLAB. In particular we applied its standard normal distribution function ``randn()"  to
generate Gaussian matrices and	 calculated 
 $\epsilon$-ranks of   matrices for 
 $\epsilon=10^{-6}$
 by using the MATLAB's function 
 "rank(-,1e-6)",
 which only counts singular values greater than $10^{-6}$.

 Our tables display the  mean value of
 the
 spectral norm of the relative 
 output error over 1000 runs for every class of inputs
 as well as the standard deviation (std) except where it is indicated otherwise.
      Some numerical experiments were executed  
      with software custom programmed in $C^{++}$ and compiled with LAPACK version 3.6.0 libraries. 
  

\subsection{Input matrices for LRA}\label{ststmtrcs}

   
 We used the following two classes of 
 input matrices $M$ for testing LRA algorithms. 

\medskip

{\bf Class I} (Synthetic inputs):  
  Perturbed $n\times n$
factor-Gaussian matrices with expected 
 rank $r$, that is, matrices $W$ in the form
$$M = G_1 * G_2 + 10^{-10} G_3,$$
for three Gaussian matrices $G_1$
of size $n\times r$, $G_2$ of size $r\times n$,
and  $G_3$
 of size $n\times n$. 
\medskip

{\bf Class II}:  The dense  matrices with smaller ratios of ``$\epsilon$-rank/$n$"  from the built-in test problems in Regularization 
Tools, which came from discretization (based on Galerkin or quadrature methods) of the Fredholm  Integral Equations of the first kind,\footnote{See 
 http://www.math.sjsu.edu/singular/matrices and 
  http://www2.imm.dtu.dk/$\sim$pch/Regutools 
  
For more details see Chapter 4 of the Regularization Tools Manual at \\
  http://www.imm.dtu.dk/$\sim$pcha/Regutools/RTv4manual.pdf } namely
to the following six input classes  from the Database:
 
\medskip

{\em baart:}       Fredholm Integral Equation of the first kind,

{\em shaw:}        one-dimensional image restoration model,
 
{\em gravity:}     1-D gravity surveying model problem,
 




{\em wing}:        problem with a discontinuous
 solution,

{\em foxgood:}     severely ill-posed problem, 
 
{\em inverse Laplace:}   inverse Laplace transformation.
    
\medskip 


\subsection{Four algorithms used}


 In our tests 
we applied and compared the following four algorithms
for computing CUR LRA to input matrices $M$ having
 $\epsilon$-rank $r$:
\begin{itemize}
\item
{\bf Tests 1 (The
Primitive
 algorithm for $k=l=r$):}
Randomly choose two index sets $\mathcal{I}$ and $\mathcal{J}$, both of cardinality $r$, then compute a nucleus 
$U=M_{\mathcal{I}, \mathcal{J}}^{-1}$ and define  CUR LRA
\begin{equation}\label{eqtsts1}
\tilde M:=CUR=M_{:, \mathcal{J}} \cdot  M_{\mathcal{I}, \mathcal{J}}^{-1} \cdot M_{\mathcal{I},\cdot}.
\end{equation} 
%
\item    
{\bf Tests 2 (Five loops of  C--A):}
Randomly choose an  initial row index set 
$\mathcal{I}_0$ of cardinality $r$, then perform five loops of C--A
 (cf. Sec. \ref{scait}) incorporating Alg. 1 of \cite{P00}
 as a subalgorithm  that
  produces $r\times r$ CUR generators.
  At the end compute a nucleus $U$ and define 
CUR LRA  as
in Tests 1.
\end{itemize} 
 
\begin{itemize}
\item
{\bf Tests 3 (A Cynical algorithm for $p=q=4r$ and $k=l=r$):}
Randomly choose a row index set $\mathcal{K}$ and a column index set $\mathcal{L}$,
both of cardinality $4r$, and then apply Algs. 1 and 2 from \cite{P00} in order
to compute a $r\times r$ submatrix $M_{\mathcal{I}, \mathcal{J}}$ of $M_{\mathcal{K}, \mathcal{L}}$. Compute a nucleus and obtain  CUR LRA by applying 
equation (\ref{eqtsts1}).

\item
{\bf Tests 4 (Combination of a single 
 C--A loop with Tests 3):}
Randomly choose a  column index set 
$\mathcal{L}$
of cardinality $4r$; then
perform a single C--A
loop
(made up of  a single horizontal  step and  a single vertical step): First by applying Alg. 1 from \cite{P00} define an index set $\mathcal{K}'$ of cardinality $4r$ and the submatrix  $M_{\mathcal{K}', \mathcal{L}}$ in  $M_{:, \mathcal{L}}$; then by applying this algorithm to matrix  $M_{\mathcal{K}',:}$ find an index set  
$\mathcal{L}'$ of cardinality $4r$ and define  submatrix  $M_{\mathcal{K}', \mathcal{L}'}$  in $M_{\mathcal{K}',:}$. Then proceed as in Tests 3 -- find an $r\times r$ submatrix $M_{\mathcal{I}, \mathcal{J}}$ in $M_{\mathcal{K}', \mathcal{L}'}$ by applying Algs. 1 and 2 from \cite{P00}, compute a nucleus and  CUR LRA.
\end{itemize}


\subsection{CUR LRA of the  matrices of class I}\label{sinteq0}

In the tests of this subsection 
we  
computed CUR LRA of factor-Gaussian matrices of
$\epsilon$-rank $r$, of class I.
 

Table \ref{tb_ranrc} shows the summary statistics of the relative error norm $\frac{||\tilde M-M||}{||M||}$ (in spectral norm) observed over 1000 runs of each of the four tests for $n =256, 512, 1024$ and $r = 8, 16, 32$. We also display the relative error norm for the best rank-$r$ approximation given by the $(r+1)$-st largest singular value to establish the baseline for the performance of our tests.  The results of Tests 1 fall in the range $[10^{-8},10^{-7}]$. The other tests show results in $[10^{-11},10^{-10}]$ with Tests 2, Tests 4, and Tests 3 in the decreasing order of accuracy. 
The column SVD of the table
displays optimal SVD-based error estimates.
On the average the error bounds of Tests 2, 4, 3, and 1 exceed  this baseline
bounds by roughly factors of 10,
15, 29, and over 10,000, respectively,  but even  the crudest CRAs output in  Test 1  were  accurate enough for some applications and for the initialization of the LRA refinement by  means of the algorithms of
\cite[Sec. 6]{TYUC17},
 \cite{TYUC19,PLa,LPSa}.
\setlength\tabcolsep{4 pt}
\renewcommand{\arraystretch}{1.2}
\begin{table}[h!]
\small
\begin{center}
\begin{tabular}{|c c|c c|c c|c c|c c|c c|}
\hline     
 &  &\multicolumn{2}{c|}{\bf SVD} &\multicolumn{2}{c|}{\bf Tests 1} & \multicolumn{2}{c|}{\bf Tests 2}&\multicolumn{2}{c|}{\bf Tests 3} & \multicolumn{2}{c|}{\bf Tests 4}\\ \hline
{\bf n} & {\bf r} & {\bf mean} & {\bf std}  &{\bf mean} & {\bf std}  & {\bf mean} & {\bf std}& {\bf mean} & {\bf std} & {\bf mean} & {\bf std}\\ \hline
256 & 8     & 1.01e-11 & 3.92e-13 & 1.60e-08 & 8.65e-08 & 5.94e-11 & 8.06e-12 & 1.13e-10 & 2.36e-11 & 8.23e-11 & 1.64e-11 \\ \hline
256 & 16   & 9.12e-12 & 3.09e-13 & 2.44e-07 & 5.85e-06 & 7.31e-11 & 1.02e-11 & 1.12e-10 & 2.02e-11 & 9.45e-11 & 1.51e-11 \\ \hline
256 & 32   & 7.80e-12 & 2.32e-13 & 4.82e-08 & 2.50e-07 & 8.93e-11 & 1.10e-11 & 1.13e-10 & 1.69e-11 & 1.04e-10 & 1.50e-11 \\ \hline
512 & 8     & 7.64e-12 & 2.23e-13 & 3.50e-08 & 2.72e-07 & 5.71e-11 & 7.08e-12 & 1.21e-10 & 2.67e-11 & 8.34e-11 & 1.62e-11 \\ \hline
512 & 16   & 7.06e-12 & 1.68e-13 & 1.18e-07 & 2.53e-06 & 7.08e-11 & 8.96e-12 & 1.26e-10 & 2.21e-11 & 9.98e-11 & 1.59e-11 \\ \hline
512 & 32   & 6.36e-12 & 1.37e-13 & 7.43e-08 & 7.47e-07 & 9.25e-11 & 1.14e-11 & 1.34e-10 & 1.92e-11 & 1.20e-10 & 1.73e-11 \\ \hline
1024 & 8   & 5.63e-12 & 1.13e-13 & 2.42e-08 & 2.46e-07 & 5.39e-11 & 5.89e-12 & 1.28e-10 & 2.83e-11 & 8.10e-11 & 1.55e-11 \\ \hline
1024 & 16 & 5.34e-12 & 9.23e-14 & 6.12e-08 & 6.91e-07 & 6.94e-11 & 7.68e-12 & 1.37e-10 & 2.35e-11 & 1.04e-10 & 1.73e-11 \\ \hline
1024 & 32 & 4.95e-12 & 7.55e-14 & 6.20e-07 & 1.36e-05 & 9.17e-11 & 1.06e-11 & 1.51e-10 & 2.09e-11 & 1.29e-10 & 1.87e-11 \\ \hline
\end{tabular}
\caption{Errors of CUR LRA of random matrices of class I}
\label{tb_ranrc}
\end{center}
\end{table}
 \renewcommand{\arraystretch}{1}
 \setlength\tabcolsep{6pt}

\subsection{LRA by means of random sampling
and C-A acceleration}
\label{ststc-alvrg}

 
 C-A  iterations can be viewed as a specialization of {\em Alternating Directions Implicit (ADI)}  method to  LRA.

Whp the randomized The algorithms of  \cite{DMM08}, run at
superlinear cost and are expected to output an LRA within any fixed positive relative error norm  bound $\epsilon$. 

 Next we cover our tests both for the randomized algorithm of \cite{DMM08}, 
 and its  combination 
  with C-A iterations.
 Tables  
 \ref{tabcadmm1} 
 and    
 \ref{tabcadmm}   
  display the  relative errors
  $ \frac{\|M - \tilde M \|}{\|M \|}$ of the LRA $\tilde M$  of  the matrices $M$ 
computed  in two ways:   by means of
  \cite[Alg. 2]{DMM08} (see the lines marked ``CUR")
 or eight C-A iterations with \cite[Alg. 1]{DMM08}  applied  at all vertical and horizontal steps
 (see the lines marked ``C-A"). The overall cost of performing the algorithms is superlinear in the former case and sublinear in the latter case.
 In almost all cases this dramatic acceleration was 
 achieved  at the price of only minor deterioration of output accuracy. 
 
  The columns of the tables marked with "$\epsilon$-rank" display  $\epsilon$-rank of an input matrix.
 The columns of the tables  marked with "$k=l$" show the number of rows and  columns in a square matrix of CUR generator. 


\begin{table}[ht]
\centering
\begin{tabular}{|c|c|c|c|c|c|c|c|}
\hline
input & algorithm & m & n & $\epsilon$-rank & k=l & mean & std \\ \hline
finite diff & C-A & 608 & 1200 & 94 & 376  & 6.74e-05 & 2.16e-05 \\ \hline
finite diff & CUR & 608 & 1200 & 94 & 376  & 6.68e-05 & 2.27e-05 \\ \hline
finite diff & C-A & 608 & 1200 & 94 & 188  & 1.42e-02 & 6.03e-02 \\ \hline
finite diff & CUR & 608 & 1200 & 94 & 188  & 1.95e-03 & 5.07e-03 \\ \hline

baart & C-A & 1000 & 1000 & 6 & 24 & 2.17e-03 & 6.46e-04 \\ \hline
baart & CUR & 1000 & 1000 & 6 & 24 & 1.98e-03 & 5.88e-04 \\ \hline
baart & C-A & 1000 & 1000 & 6 & 12 & 2.05e-03 & 1.71e-03 \\ \hline
baart & CUR & 1000 & 1000 & 6 & 12 & 1.26e-03 & 8.31e-04 \\ \hline
baart & C-A & 1000 & 1000 & 6 & 6  & 6.69e-05 & 2.72e-04 \\ \hline
baart & CUR & 1000 & 1000 & 6 & 6  & 9.33e-06 & 1.85e-05 \\ \hline

shaw & C-A & 1000 & 1000 & 12 & 48 & 7.16e-05 & 5.42e-05 \\ \hline
shaw & CUR & 1000 & 1000 & 12 & 48 & 5.73e-05 & 2.09e-05 \\ \hline
shaw & C-A & 1000 & 1000 & 12 & 24 & 6.11e-04 & 7.29e-04 \\ \hline
shaw & CUR & 1000 & 1000 & 12 & 24 & 2.62e-04 & 3.21e-04 \\ \hline
shaw & C-A & 1000 & 1000 & 12 & 12 & 6.13e-03 & 3.72e-02 \\ \hline
shaw & CUR & 1000 & 1000 & 12 & 12 & 2.22e-04 & 3.96e-04 \\ \hline
\end{tabular}
\caption{LRA errors of Cross-Approximation (C-A)   incorporating 
\cite[Algorithm 1]{DMM08} in comparison to stand-alone CUR of 
\cite[Algorithm 2]{DMM08}.}
\label{tabcadmm1}
\end{table} 

\begin{table}[ht]
\centering
\begin{tabular}{|c|c|c|c|c|c|c|}
\hline
input & algorithm & m = n & $\epsilon$-rank & $k=l$ & mean & std \\ \hline
foxgood & C-A & 1000  & 10 & 40 & 3.05e-04 & 2.21e-04 \\\hline
foxgood & CUR & 1000  & 10 & 40 & 2.39e-04 & 1.92e-04 \\ \hline
foxgood & C-A & 1000  & 10 & 20 & 1.11e-02 & 4.28e-02 \\ \hline
foxgood & CUR & 1000 & 10 & 20 & 1.87e-04 & 4.62e-04 \\ \hline
wing & C-A & 1000  & 4 & 16 & 3.51e-04 & 7.76e-04 \\ \hline
wing & CUR & 1000  & 4 & 16 & 2.47e-04 & 6.12e-04 \\ \hline
wing & C-A & 1000  & 4 & 8  & 8.17e-04 & 1.82e-03 \\ \hline
wing & CUR &  1000 & 4 & 8  & 2.43e-04 & 6.94e-04 \\ \hline
wing & C-A &  1000 & 4 & 4  & 5.81e-05 & 1.28e-04 \\ \hline
wing & CUR &  1000 & 4 & 4  & 1.48e-05 & 1.40e-05 \\ \hline

gravity & C-A &  1000 & 25 & 100 & 1.14e-04 & 3.68e-05 \\ \hline
gravity & CUR &  1000 & 25 & 100 & 1.41e-04 & 4.07e-05 \\ \hline
gravity & C-A &  1000 & 25 & 50  & 7.86e-04 & 4.97e-03 \\ \hline
gravity & CUR &  1000 & 25 & 50  & 2.22e-04 & 1.28e-04 \\ \hline
inverse Laplace & C-A &  1000 & 25 & 100 & 4.15e-04 & 1.91e-03 \\ \hline
inverse Laplace & CUR &  1000 & 25 & 100 & 5.54e-05 & 2.68e-05 \\ \hline
inverse Laplace & C-A &  1000 & 25 & 50  & 3.67e-01 & 2.67e+00 \\ \hline
inverse Laplace & CUR &  1000 & 25 & 50 &  2.35e-02 & 1.71e-01 \\ \hline
\end{tabular}
\caption{LRA errors of Cross-Approximation (C-A)   incorporating 
\cite[Algorithm 1]{DMM08}  in comparison to stand-alone CUR of 
\cite[Algorithm 2]{DMM08}.} \label{tabcadmm}
\end{table}


 
  





\clearpage


{\bf \Large Appendix} 
\appendix



\section{Small families of hard inputs for 
superfast LRA}\label{shrdin}

  Any sublinear cost LRA algorithm
  fails on the following small families of LRA inputs.
  
\begin{example}\label{exdlt} 
 Let  $\Delta_{i,j}$ denote an $m\times n$ matrix
 of rank 1  filled with 0s except for its $(i,j)$th entry filled with 1. The $mn$ such matrices $\{\Delta_{i,j}\}_{i,j=1}^{m,n}$ form a family of  $\delta$-{\em matrices}.
We also include the $m\times n$ null matrix $O_{m,n}$
filled with 0s  into this family.
Now fix any sublinear cost  algorithm; it does not access the $(i,j)$th  
entry of its input matrices  for some pair of $i$ and $j$. Therefore, it outputs the same approximation 
of the matrices $\Delta_{i,j}$ and $O_{m,n}$,
with an undetected  error at least 1/2.
Arrive at the same conclusion by applying the same argument to the
set of $mn+1$ small-norm perturbations of 
the matrices of the above family and to the                                                                                                                                   
 $mn+1$ sums  
of the latter matrices with  any
   fixed $m\times n$ matrix of low rank.
Finally, the same argument shows that 
a posteriori estimation of the output errors of an LRA algorithm
applied to the same input families
cannot run at  sublinear cost. 
\end{example}

The example actually covers randomized LRA algorithms as well. Indeed, suppose that an LRA algorithm does not involve 
 a fixed entry of an input matrix with a  probability $p>0$.
Apply this algorithm to two matrices
of low rank whose difference at  this entry is equal to a large constant $C$.   
Then
with a  
probability $p$ the algorithm  
has error at least $C/2$  at this entry for a least one of these two matrices. 

\section{Abridged SRHT  matrices}\label{ssrht}

  
With  {\em sparse subspace embedding} \cite{C16,CDDRa,CFSa},
\cite[Sec. 3.3]{TYUC19},
\cite[Sec. 9]{MT20} one obtains significant acceleration but still does not yield superfast algorithms. 
According to \cite{L09},
 such acceleration tends to
 make the accuracy of output LRAs somewhat less reliable, although \cite{CFSa} partly overcomes
 this  problem for incoherent matrices. One can  multiply a  matrix by Subsampled Randomized Hadamard or Fourier Transform ({\em SRHT or SRFT)} dense matrices
 towards  incoherence \cite{CFSa}, but this step is not superfast.

For a compromise, we devise {\em Abridged SRHT multipliers}. They are  sparse, can  be multiplied by a dense matrix at sublinear cost, and the paper \cite{PLSZb} 
studied  their
application to LRA,  both formally and empirically,
and similarly defined and studied Abridged SRFT multipliers.

We  proceed
  by  means of abridging 
    the classical 
recursive processes of the generation  of $n\times n$ 
SRHT  matrices,
  obtained from   the  $n\times n$ dense
matrices $H_n$ of  Walsh-Hadamard transform for $n=2^t$
(cf. \cite[Sec. 3.1]{M11}).
The $n\times n$ matrices $H_n$ are obtained in $t=\log_2(n)$ recursive  steps, but we 
 only perform $d\ll t$
 steps, and the resulting 
 abridged matrix $H_{d,d}$ can be multiplied by a vector by using $2dn$   additions and subtractions.
 SRHT matrices are obtained
 from the matrices $H_n$ by means of random sampling and scaling, which we also apply to the $d$-Abridged Hadamard  transform matrices  $H_{d,d}$. They turn into $H_n$ for $d=t$ but are sparse for $d\ll t$. Namely, we
   write $H_{d,0}:=I_{n/2^d}$ and then
specify the following recursive process: 
 

\begin{equation}\label{eqrfd}
H_{d,0}:=I_{n/2^d},~
H_{d,i+1}:=\begin{pmatrix}
H_{d,i} & H_{d,i} \\
H_{d,i} & -H_{d,i}
  \end{pmatrix}
  ~{\rm for}~i=0,1,\dots,d-1, 
\end{equation}

For any fixed pair of $d$ and $i$, 
each of the matrices  
 $H_{d,i}$ 
is orthogonal 
up to scaling and
 has $2^d$ nonzero entries 
in every row and  column;
  we can compute the product 
 $MH$ for an   $n\times k$ submatrix $H$ of $H_{d,d}$ by using less than $km2^d$ additions and subtractions.


Now define the 
$d$-Abridged Scaled and Permuted 
 Hadamard matrices, $PDH_{d,d}$, where $P$ is a
 random 
 sampling 
 matrix and  $D$ is the matrix of  
random integer diagonal scaling.
Each random permutation or scaling  
 contributes up to $n$ random parameters.  
We can involve more random parameters by applying random permutation and scaling 
also to some or all
 intermediate matrices $H_{d,i}$ for $i=0,1,\dots,d$.

The first $k$ columns of $H_{d,d}$  for  
 $r\le k\le n$ form a
 $d$-Abridged  
 SRHT matrix $H$, which turns into
an SRHT matrix for $d=t$, where $k=r+p$, $r$ is a target rank and $p$ is the oversampling parameter
 (cf. \cite[Sec. 11]{HMT11}).

\section{Volume of a factor-Gaussian matrix}\label{svlmfct}


In this section we assume dealing with a random input matrixces having the distribution of a  two-sided factor-Gaussian matrix with expected rank $r$ and then estimate its
$r$-projective volume.  We can  extend this estimate to its sufficiently small neighborhood. 
 
For an $m\times r$ Gaussian matrix $G$,  $r \le m$, and $r$ independent $\chi^2$ random variables $\chi^2_{m-i+1}$ with $m-i+1$ degrees of freedom,
  $i=1,\dots,r$, recall  that
   $$ v_2(G)^2 
   \sim \prod_{i=1}^{r} \chi^2_{m-i+1}.$$
  
  Recall two auxiliary results,  about  concentration of $\chi^2$ random variables and about  statistical order of $r\cdot{\rm vol}(G)^{1/r}$ and $\chi^2$ random variables with appropriate degrees of freedom, respectively.
 
\begin{lemma}[adapted from {\cite[Lemma 1]{LM00}}]\label{lemma:laurent} Let
 $Z\sim \chi^2_k$ and let $r$ be an integer. Then
   $$ \prob\Big\{ \frac{Z}{r} \ge 1 + \theta \Big\} \le \exp\Big(-\frac{\theta r}{4}\Big)~
{\rm for~any}~\theta > 4;$$ 
$$    \prob\Big\{ \frac{Z}{r} \le 1 - \phi \Big\} \le 
    \exp\Big(-\frac{\phi^2 r}{4}\Big)~
{\rm for~any}~\phi > 0.$$
\end{lemma}

\begin{theorem}{\rm\cite[Theorem 2]{MZ08}.}\label{thm:magen}
Let $m \ge r \ge 2$ and let $G$ be an $m\times r$ Gaussian matrix. Then 
    $$\chi^2_{r(m-r+1) + \frac{(r-1)(r-2)}{2}} \succeq r{ v}_2(G)^{2/r} \succeq \chi^2_{r(m-r+1)}$$
greater than or equal to $B$.
\end{theorem}

Next  estimate the volume of a  Gaussian matrix based on the above results.


\begin{lemma}\label{lemma:gau_vol_con}
Let $G$ be an $m\times r$ Gaussian matrix for $m \ge r \ge 2$. Then 
   $$ \prob\big\{ 
   v_{2,r}(G) \ge (1+\theta)^{r/2} (m-r/2)^{r/2}
    \big\} \le \exp{\big(-\frac{\theta}{4}(mr-\frac{r^2}{2} - \frac{r}{2} + 1) \big)}
~{\rm for}~\theta > 4;$$
$$    \prob\big\{ 
v_{2,r}(G) \le (1-\phi)^{r/2} (m-r+1)^{r/2}
    \big\} 
    \le\exp{\Big(-\frac{\phi^2}{4}r(m-r+1) 
    \Big)}~{\rm for}~\phi > 0.$$
\end{lemma}

\begin{proof}
Combine Lemma \ref{lemma:laurent} and Theorem \ref{thm:magen}.
\end{proof} 

We also need the following result
of \cite{OZ18}.


\begin{theorem}\label{thvolfctrg} 
For $G\in \mathbb  R^{m\times q}$,  $H\in \mathbb  R^{q\times n}$,  and $1\le r\le q$, it holds that
 $v_{2,r}(GH)\le v_{2,r}(G)v_{2,r}(H)$. 
\end{theorem}

Next assume  that  $\min\{m, n\big\} \gg r$ and that $p$ and $q$ are two sufficiently large integers and  then  prove that the volume of any fixed $p\times q$ submatrix of an $m\times n$  two-sided
factor-Gaussian matrix with expected rank $r$  has a reasonably large lower bound
whp. 
\begin{theorem}\label{thm:gau_vol_large}
Let $W = G\Sigma H$ be an $m\times n$  two-sided factor-Gaussian matrix with expected rank  $r \le \min\{m, n\}$. Let $\mathcal I$ and $\mathcal J$ be row and column index sets such that $|\mathcal I| = p \ge r$ and $|\mathcal J| = q \ge r$. Let $\phi$ be a positive number.  Then 
 $$   
v_{2,r}{\big( W_{\mathcal I, \mathcal J}\big)} \ge (1-\phi)^r(p-r+1)^{r/2}(q-r+1)^{r/2}
v_{2,r}\big(\Sigma\big) $$
with a probability no less than $1 - \exp{\big(\frac{\phi^2}{4}r(p-r+1)\big)} - \exp{\big(\frac{\phi^2}{4}r(q-r+1)\big)}$.
\end{theorem}

\begin{proof}
Recall Theorem \ref{thvolfctrg} and obtain
  $$  
  v_{2,r}{\big(W_{\mathcal I, \mathcal J}\big)} = 
  v_{2,r}(G_{\mathcal I, :})
v_{2,r}(\Sigma)
  v_{2,r}(H_{:, \mathcal J}),$$
where $G_{\mathcal I, :}$ and $H_{:, \mathcal J}$ are independent Gaussian matrices. Complete the proof by 
applying Lemma \ref{lemma:gau_vol_con} and the Union Bound. 

\end{proof}


Extend this theorem by estimating the volume of a  two-sided factor-Gaussian matrix.

Due to the  volume concentration of a Gaussian matrix,  
it is unlikely that the maximum volume of a matrix
 in a set of
 moderate number of Gaussian matrices 
greatly exceeds the volume 
of a fixed matrix in this set. Based on this observation, we arrive at weak  maximization 
of the volume of any fixed submatrix of a   two-sided factor-Gaussian matrix.

\begin{lemma}\label{lemma:gau_vol_maximal}
Let $G_1, G_2, \dots, G_M$ be a collection of $M$ Gaussian matrices of size $m\times r$, for $m \ge r$. Then 
whp specified in the proof we have
\begin{equation}\label{ineq:gau_vol_maximal}
    \frac{\max_{1\le i\le M} \big(
    v_{2,r}(G_i)\big)}{
    v_{2,r}(G_1)} \le \Big( \frac{(1+\theta)(1+r/m)}{1-\phi} \Big)^{r/2}.
\end{equation}
\end{lemma}

\begin{proof}
Write $V_{max} := \max_{1\le i\le M} \big( 
v_{2,r}(G_i)\big)$. Combine Lemma \ref{lemma:gau_vol_con} and the Union Bound  to obtain 
  $$  \prob \big\{  V_{max}
    \ge  (1+\theta)^{r/2} (m-r/2)^{r/2}\big\}
    \le M\cdot\exp{\Big(-\frac{\theta}{4}\Big(mr-\frac{r^2}{2} - \frac{r}{2} + 1\Big) \Big)} $$
for $\theta > 4$.
Moreover,  
 $$   \prob\big\{
 v_{2,r}(G_1) \le (1-\phi)^{r/2} (m-r+1)^{r/2}
    \big\} 
    \le\exp{\big(-\frac{\phi^2}{4}r(m-r+1) 
    \big)} $$
for $\phi > 0$. 
Now assume that 
 $m > 2r$ and  readily deduce that
$$\frac{m - r/2}{ m - r + 1} < 1 + \frac{r}{m}.$$
Combine these results and obtain that inequality (\ref{ineq:gau_vol_maximal}) holds
with a probability no less than \\ $1 - M\cdot\exp{\big(-\frac{\theta}{4}(mr-\frac{r^2}{2} - \frac{r}{2} + 1) \big)} - \exp{\big(-\frac{\phi^2}{4}r(m-r+1)\big)}$.
\end{proof}

\begin{remark}\label{remark:prob_ineq}
The exponent  $r/2$ in the volume ratio may be disturbing but is natural because $
v_{2,r}(G_i)$ is essentially the volume of an $r$-dimensional parallelepiped, and difference in each dimension will contribute to the difference in the volume. The impact of factor $M$ 
on the probability estimates can be mitigated with parameter $m$, that is, the probability is high and even close to 1 if $m$ is set sufficiently large. Namely, let 
\begin{equation*} 
    m \ge 1 + r + \frac{4\ln M}{r\theta}.
\end{equation*}
Then  we  readily deduce that
\begin{equation*}
    M\cdot\exp\Big(
    -\frac{\theta}{4}\Big(mr-\frac{r^2}{2} - \frac{r}{2} + 1\Big)
    \Big) < \exp \Big(-\frac{\theta}{4}r\Big);
\end{equation*}

\begin{equation*}
    \exp{\Big(-\frac{\phi^2}{4}r(m-r+1)\Big)} < \exp\Big( -\frac{\phi^2}{2}r\Big). 
\end{equation*}
\end{remark}

\begin{theorem}\label{thm:factor_gau_locally_max}
Let $W = G\Sigma H $ be an $m\times n$ two-sided factor-Gaussian matrix with expected rank $r$ and let $r < \min\big(m, n\big)$. Let $\mathcal I$ and $\mathcal J$ be row and column index sets such that 
$|\mathcal I| = p > 2r$ and 
$|\mathcal J| = q > 2r$. Let $\theta > 4$ and $\phi > 0$ be two parameters, and further assume that 
$p \ge 1 + r + \frac{4\ln m^2/4}{r\theta}$ and 
$q \ge 1 + r + \frac{4\ln n^2/4}{r\theta}$. Then 
  the submatrix $W_{\mathcal I, \mathcal J}$ has $\big(\frac{1+\theta}{1-\phi}\big)^r\big(\frac{(p+r)(q+r)}{pq}  \big)^{r/2}$-locally maximal $r$-projective volume
with a probability no less than $1 - 2\exp{\big(-\frac{\theta}{4}r \big)}
-2\exp{\big(-\frac{\phi^2}{2}r\big)}$.
\end{theorem}

\begin{proof}
There are 
$p(m-p) \le m^2/4$ 
submatrices  of $G$ 
of size $p\times r$  that differ from $G_{\mathcal I, :}$ only in a single  row;
 likewise  there are 
$q(n-q) \le n^2/4$ 
submatrices of $H$  of size $q\times r$ that differ from $H_{:, \mathcal J}$ only in a single column. 

Now let $\mathcal I'$ and $\mathcal J'$ be any pair of row and column index sets  that  differ from $\mathcal I$ and $\mathcal J$ by a single index, respectively. 
Then Lemma \ref{lemma:gau_vol_maximal} and Remark \ref{remark:prob_ineq}
together imply that
\begin{equation}\label{ineq:factor_gau_locally_max}
    \frac{
    v_{2,r}(G_{\mathcal I', :})}{
   v_{2,r} (G_{\mathcal I, :})} \le \Big( \frac{(1+\theta)(1+\frac{r}{p})}{1-\phi} \Big)^{r/2}   
    ~\textrm{ and }~
    \frac{v_{2,r}(H_{:, \mathcal J'})}{
    v_{2,r}(H_{:, \mathcal I})} \le \Big( \frac{(1+\theta)(1+\frac{r}{q})}{1-\phi} \Big)^{r/2}     
\end{equation}
with a probability no less than
$1 - 2\exp{\big(-\frac{\theta}{4}r \big)}
-2\exp{\big(-\frac{\phi^2}{2}r\big)}$.

Recall that 
\begin{equation*}
    v_{2,r}\big( W_{\mathcal I, \mathcal J} \big)
    = v_{2,r} \big( G_{\mathcal I, :}\big)
    v_{2,r}(\Sigma)
    v_{2,r} \big(H_{:, \mathcal J}\big),
\end{equation*}
and similarly for $W_{\mathcal I', \mathcal J'}$. Inequality (\ref{ineq:factor_gau_locally_max}) implies that 
\begin{equation*}
    \frac{
    v_{2,r}\big( W_{\mathcal I', \mathcal J'}\big)}{
    v_{2,r}\big( W_{\mathcal I, \mathcal J}\big)} \le \Big(\frac{1+\theta}{1-\phi}\Big)^r\Big(\frac{(p+r)(q+r)}{pq} \Big)^{r/2}.
\end{equation*}
\end{proof} 


 
\medskip


\noindent {\bf Acknowledgements:}
Our work has been supported by NSF Grants 
CCF--1116736,
 CCF--1563942 and CCF--1733834
and PSC CUNY Award  66720-00 54.
We are also grateful to E. E. Tyrtyshnikov for the challenge
of formally supporting empirical power of C--A iterations,
to N. L. Zamarashkin for his comments on his work with A. Osinsky  on LRA via volume maximization and on the first drafts of  \cite{PLSZ17} and the present paper,   and to
  S. A. Goreinov, 
 I. V. Oseledets, A. Osinsky,  E. E. Tyrtyshnikov, and N. L. Zamarashkin for  reprints and pointers to relevant bibliography.

\clearpage



\begin{thebibliography}{hspace{0.5in}}



\bibitem{ALS24}
Kenneth Allen, Ming-Jun Lai,  Zhaiming Shen,
Maximal Volume Matrix Cross Approximation for Image
 Compression and Least Squares Solution,  {\em Advances in Computational Mathematics}, {\bf 5}, 2024. DOI: 10.1007/s10444-024-10196-7.
Also arXiv:2309.1740,   December, 2024.


\bibitem{B00}
M. Bebendorf, Approximation of Boundary Element Matrices, {\em Numer. Math.},
{\bf 86,~4}, 565--589, 2000.


\bibitem{B11}
M. Bebendorf,
Adaptive Cross Approximation of Multivariate Functions,
{\em Constructive approximation}, {\bf 34,~2}, 149--179, 2011.



\bibitem{B15}
A. Bj{\"o}rk, {\em Numerical Methods in Matrix Computations}, Springer, New York, 2015.


\bibitem{BG06}
M. Bebendorf, R. Grzhibovskis, Accelerating Galerkin BEM for 
linear elasticity using adaptive cross approximation, 
{\em Math. Methods Appl. Sci.}, {\bf 29}, 1721–-1747, 2006.


\bibitem{BR03}
M. Bebendorf, S. Rjasanow,
Adaptive Low-Rank Approximation of Collocation Matrices,
{\em Computing}, {\bf 70,~1}, 1--24, 2003.


\bibitem{BV88}
W. Bruns, U. Vetter, {\em Determinantal Rings, 
Lecture Notes in Math.,} {\bf 1327}, 
Springer,
Heidelberg, 
1988.


\bibitem{BW18}
A. Bakshi, D. P. Woodruff: Sublinear Time Low-Rank Approximation of
Distance Matrices, {\em Procs.  32nd Intern. Conf. Neural Information Processing Systems (NIPS'18),} 3786--3796, Montréal, Canada, 2018.


\bibitem{C16}
Michael B. Cohen, Nearly tight oblivious subspace embeddings by trace inequalities,
27th ACM-SIAM Symp. on Discrete Algorithms (SODA 2016), Arlington,  278 --287,  2016.
doi:10.1137/1.9781611974331.ch21.8


\bibitem{CD05}
Z. Chen, J. J. Dongarra, Condition Numbers of Gaussian Random Matrices,
{\em SIAM. J. on Matrix Analysis and Applications}, {\bf 27}, 603--620, 2005.


\bibitem{CD13}
Jiawei Chiu, Laurent Demanet, Sublinear randomized algorithms for skeleton decompositions, 
{\em SIAM J. Matrix Anal. Appl.}, {\bf 34}, 1361--1383 (2013),\\ https://doi.org/10.1137/110852310. Also
arXiv 1110.4193 Oct 2011.


\bibitem{CDDRa}
Shabarish Chenakkod, Michał Derezi´nski, Xiaoyu Dong,  Mark Rudelson, Optimal embedding dimension for sparse subspace embeddings,  arXiv:2311.10680 (2023), revised June 2024.


\bibitem{CETW23}
Yifan Chen, Ethan N. Epperly,
Joel A. Tropp, Robert J. Webber,
Randomly pivoted Cholesky:
Practical approximation of a kernel matrix
with few entry evaluations, arXiv 2207.06503, December 2023,  
last revised 22 Oct 2024.


\bibitem{CFSa}
Coralia Cartis, Jan Fiala,  Zhen Shao, Hashing embeddings of optimal dimension, with applications to linear least squares, arXiv:2105.11815 (2021).


\bibitem{CK20}
Cortinovis A., Kressner D., Low-Rank Approximation in the Frobenius Norm by Column and Row Subset Selection, {\em SIAM Journal on Matrix Analysis and Applications}, {\bf 41, 4}, 1651-1673, 2020. Also
arXiv:1908.06059.


\bibitem{CLO16}
C. Cichocki, N. Lee, I. Oseledets,
 A.-H. Phan, Q. Zhao and D. P. Mandic,
``Tensor Networks for Dimensionality Reduction 
and Large-scale Optimization.
Part 1: Low-Rank Tensor Decompositions",
Foundations and Trends® in Machine Learning:
{\bf 9, 4-5}, 249--429, 2016. 
http://dx.doi.org/10.1561/2200000059 




\bibitem{CY25} Alice Cortinovis, Lexing Ying, A Sublinear-Time Randomized Algorithm for Column and Row Subset Selection Based on Strong Rank-Revealing QR Factorizations, {\em SIAM Journal on Matrix Analysis and Applications},
 {\bf 46, 1}, 22-44 (2025)\\
   https://doi.org/10.1137/24M164063X. 
 Also arXiv 2402.13975 February 2024.


\bibitem{DMM08}
P. Drineas, M.W. Mahoney, S. Muthukrishnan, Relative-error CUR Matrix Decompositions, {\em SIAM Journal on Matrix Analysis and Applications}, {\bf 30,~2}, 844--881,  2008.


\bibitem{DS01}
K. R. Davidson, S. J. Szarek, 
Local Operator Theory, Random Matrices, and Banach Spaces, 
in {\em Handbook on the Geometry of  Banach Spaces} 
(W. B. Johnson and J. Lindenstrauss editors), pages 317--368, 
North Holland, Amsterdam, 2001. 


\bibitem{E88}
A. Edelman, Eigenvalues and Condition Numbers of Random Matrices,
{\em SIAM J. on Matrix Analysis and Applications}, {\bf 9}, {\bf 4},
543--560, 1988.


\bibitem{ES05}
A. Edelman, B. D. Sutton,  Tails of Condition Number Distributions,
{\em SIAM J. on Matrix Analysis and Applications}, {\bf 27}, {\bf 2},
547--560, 2005.


\bibitem{GE96}
M. Gu, S.C. Eisenstat, 
An Efficient Algorithm for Computing a Strong Rank Revealing QR Factorization, 
{\em SIAM J. Sci. Comput.}, {\bf 17}, 848--869, 1996.


\bibitem{GL13}
G. H. Golub, C. F. Van Loan, 
{\em Matrix Computations},
The Johns Hopkins University Press, Baltimore, Maryland, 2013 (fourth edition).


\bibitem{GOSTZ10}
S. Goreinov, I. Oseledets, D. Savostyanov, E. Tyrtyshnikov, N. Zamarashkin,
How to Find a Good Submatrix, in 
{\em Matrix Methods: Theory, Algorithms, Applications} 
(dedicated to the Memory of Gene Golub, edited by V. Olshevsky and E. Tyrtyshnikov), 
pages 247--256,
World Scientific Publishing, New Jersey, ISBN-13 978-981-283-601-4, ISBN-10-981-283-601-2, 
 2010.


\bibitem{GT01}
S. A. Goreinov, E. E. Tyrtyshnikov,
The Maximal-Volume Concept in
Approximation by Low Rank Matrices,
{\em  Contemporary Mathematics}, {\bf 208},
47--51, 2001.


\bibitem{GT11}
S. A. Goreinov, E. E. Tyrtyshnikov,
Quasioptimality of Skeleton Approximation of a Matrix 
on the Chebyshev Norm,
{\em Russian Academy of Sciences: Doklady, 
Mathematics}
{\em (DOKLADY AKADEMII NAUK)},
{\bf 83,~3}, 1--2, 2011.


\bibitem{GTZ97}
 S. A. Goreinov, E. E. Tyrtyshnikov,  N. L. Zamarashkin, 
A Theory of Pseudo-skeleton Approximations,
{\em Linear Algebra and Its Applications}, {\bf 261}, 1--21, 1997.


\bibitem{GZT95}
 S. A. Goreinov, N. L.~Zamarashkin, E.~E.~Tyrtyshnikov,
Pseudo-skeleton approximations,
{\em Russian Academy of Sciences: Doklady, Mathematics} {\em (DOKLADY AKADEMII NAUK)},
{\bf 343,~2}, 151--152, 1995.


\bibitem{GZT97}
 S. A. Goreinov, N. L.~Zamarashkin, E.~E.~Tyrtyshnikov,
Pseudo-skeleton Approximations by Matrices of Maximal Volume,
{\em Mathematical Notes}, {\bf 62,~4}, 515--519, 1997.


\bibitem{HMT11}
N. Halko, P. G. Martinsson, J. A. Tropp,
Finding Structure with Randomness: Probabilistic Algorithms
for Constructing
 Approximate Matrix Decompositions, 
{\em SIAM Review}, {\bf 53,~2}, 217--288, 2011.


\bibitem{KS17}
N. Kishore Kumar, J. Schneider,
Literature Survey on Low Rank Approximation of Matrices,
{\em Linear and Multilinear Algebra}, 
{\bf 65,~11}, 2212--2244, 2017, and
arXiv:1606.06511v1 [math.NA] 21 June 2016.


\bibitem{L09}
E. Liberty, Accelerated Dense Random Projections, PhD Thesis, Yale Univ., 2009.


\bibitem{LM00}
B. Laurent, P. Massart,
Adaptive estimation of a quadraticfunctional by model selection,
{\em Annals of Statistics}, 1302--1338, 2000.


\bibitem{LP20}
 Q. Luan, V. Y. Pan, CUR LRA at Sublinear Cost Based on Volume Maximization. In {\em LNCS} {\bf 11989}, {\em Book: Mathematical Aspects of Computer and Information Sciences (MACIS 2019)}, D. Salmanig et al (Eds.), Springer Nature Switzerland AG 2020, Chapter No: {\bf 10}, pages 1--17, Springer Nature Switzerland AG 2020 Chapter DOI:10.1007/978-3-030-43120-4\_10
   

\bibitem{LPSa}
Q. Luan, V. Y. Pan, J. Svadlenka,
Low Rank Approximation  Directed by Leverage Scores 
and Computed at Sub-linear Cost, arXiv:1906.04929 
 ( 10 Jun 2019).


\bibitem{LYMHY}
Y. Li, H. Yang, E. R. Martin, K. L. Ho,  L. Ying, Butterfly factorization, {\em Multiscale
Model. Simul.}, {\bf 13},  714--732 (2015) https://doi.org/10.1137/15M1007173.
 Also arXiv:1502.01379
February 2015


\bibitem{M11}
M. W. Mahoney,
Randomized Algorithms for Matrices and Data,  
{\em Foundations and Trends in Machine Learning}, NOW Publishers, {\bf 3,~2}, 2011. Preprint: arXiv:1104.5557 (2011)
(Abridged version in: {\em Advances in Machine Learning and Data Mining for Astronomy}, 
edited by M. J. Way et al., pp. 647--672, 2012.) 


\bibitem{MD09}
M. W. Mahoney,  P. Drineas,
CUR matrix decompositions for improved data analysis,
{\em Proceedings of the National Academy of Sciences},
{\bf 106~3}, 697--702, 2009.


\bibitem{MT20} 
Per-Gunnar Martinsson, Joel A. Tropp, Randomized numerical linear algebra: Foundations and algorithms, {\em Acta Numerica}, {\bf 29,} 403--572  (2020)
 

\bibitem{MW17}
Cameron Musco, D. P. Woodruff: Sublinear Time Low-Rank Approximation of
Positive Semidefinite Matrices, 
{\em IEEE 58th Annual Symposium on Foundations of Computer Science (FOCS),}
 672--683, 2017.


\bibitem{MZ08}
A. Magen, A. Zouzias,
Near optimal dimensionality reductions that preserve volumes,
{\em Approximation, Randomization and Combinatorial Optimization. Algorithms and Techniques},
523 -- 534, 2008.


\bibitem{N20}
Yuji Nakatsukasa,
Fast and stable randomized low-rank matrix approximation,
arXiv:2009.11392 (Sept 2020)
DOI:10.48550/arXiv 2009.11392


\bibitem{O18}
A.I. Osinsky,
Rectangular Matrix Volume and Projective Volume Search Algorithms,
arXiv:1809.02334, September 17, 2018.
  

\bibitem{OZ18}
A.I. Osinsky, N. L. Zamarashkin,
 Pseudo-skeleton Approximations
 with Better Accuracy Estimates,
{\em Linear Algebra and Its Applications}, {\bf 537}, 221--249, 2018.


\bibitem{P00}
C.-T. Pan,
On the Existence and Computation of
Rank-Revealing LU Factorizations,
{\em Linear Algebra and its Applications}, {\bf 316}, 199--222, 2000.
 

\bibitem{P01}
V. Y. Pan,
{\em Structured Matrices and Polynomials: Unified Superfast Algorithms}, \\
Birk\-h\"auser/Sprin\-ger, Boston/New York, 2001. 


\bibitem{P15}
V.Y. Pan, 
Transformations of Matrix Structures Work Again. 
Linear Algebra and Its Applications {\bf 465}, 
1-32, 2015.
doi: 10.1016/j.laa.2014.09.004


\bibitem{PLa}
V. Y. Pan, Q. Luan, 
Refinement of Low Rank
Approximation of a Matrix
at Sub-linear Cost, arXiv:1906.04223 (Submitted on 10 Jun 2019).


\bibitem{PLSZ16}
V. Y. Pan, Q. Luan, 
J. Svadlenka, L.Zhao,
Primitive and Cynical Low Rank Approximation, Preprocessing and Extensions,
arXiv 1611.01391 (Submitted on 3 November,
2016).


\bibitem{PLSZ17}
V. Y. Pan, Q. Luan, J. Svadlenka, L. Zhao,
Superfast Accurate Low Rank Approximation, preprint,
arXiv:1710.07946 (Submitted on  22 October, 
2017).



\bibitem{PLSZa}
V. Y. Pan, Qi Luan, John Svadlenka, Liang Zhao, CUR Low Rank Approximation at  Sublinear Cost, arXiv:1906.04112 v1, June 2019.
 

\bibitem{PLSZb}
V. Y. Pan, Q. Luan, J. Svadlenka, L. Zhao,
Low Rank Approximation at Sub-linear Cost by Means of Subspace Sampling, arXiv:1906.04327  (Submitted on 10 Jun 2019).



\bibitem{PLSZ20}
V. Y. Pan, Q. Luan, J. Svadlenka, L. Zhao, 
 Sublinear Cost Low Rank Approximation via Subspace Sampling, In {\em LNCS} {\bf 11989}, {\em Book: Mathematical Aspects of Computer and Information Sciences (MACIS 2019)}, D. Salmanig et al (Eds.), Springer Nature Switzerland AG 2020, Chapter No: {\bf 9}, pages 1--16, Springer Nature Switzerland AG 2020 Chapter DOI:10.1007/978-3-030-43120-4\_9



\bibitem{PQY15} 
V. Y. Pan, G. Qian, X. Yan, 
Random Multipliers Numerically Stabilize Gaussian and Block Gaussian Elimination: 
Proofs and an Extension to Low-rank Approximation,
{\em Linear Algebra and Its Applications}, {\bf 481}, 202--234, 2015.


\bibitem{PZ17a}
V. Y. Pan, L. Zhao,
New Studies of Randomized Augmentation and Additive Preprocessing,
{\em  Linear Algebra and Its Applications}, {\bf 527}, 256--305, 2017. \\
 http://dx.doi.org/10.1016/j.laa.2016.09.035. 


\bibitem{PZ17b}
V. Y. Pan, L. Zhao,
Numerically Safe Gaussian Elimination
with No Pivoting,
{\em  Linear Algebra and Its Applications}, {\bf 527}, 
349--383, 2017. \\
http://dx.doi.org/10.1016/j.laa.2017.04.007. Also
arxiv 1501.05385


\bibitem{RV09}
M. Rudelson, R. Vershynin, Smallest 
Singular Value of a Random Rectangular Matrix, {\em Comm. Pure Appl. Math.}, {\bf 62,~12}, 1707--1739, 2009.  \\
https://
doi.org/10.1002/cpa.20294
 


\bibitem{S16}
V. Simoncini,   Computational Methods for Linear Matrix Equations, {\em SIAM Review}, {\bf 58 (3)} 377 -- 441, 2016.
doi:10.1137/130912839





\bibitem{SST06}
A. Sankar, D. Spielman, S.-H. Teng, 
Smoothed Analysis of the Condition Numbers and Growth Factors of Matrices, 
 {\em SIAM J. Matrix Anal. Appl.}, {\bf 28}, {\bf 2}, 446--476, 2006. 


\bibitem{T96}
E. E. Tyrtyshnikov,
 Mosaic-Skeleton Approximations,
{\em Calcolo}, {\bf 33,~1}, 47--57, 1996.


\bibitem{T00}
E. E. Tyrtyshnikov, Incomplete Cross-Approximation in the Mosaic-Skeleton Method,  
{\em Com\-put\-ing}, {\bf 64}, 367--380, 2000.


\bibitem{TWa}
Joel A. Tropp, Robert J. Webber, Randomized algorithms for low-rank matrix approximation: Design, analysis, and applications,  arXiv:
2306.12418 (2023).


\bibitem{TYUC17}
J. A. Tropp, A. Yurtsever, M. Udell, V. Cevher, Practical Sketching Algorithms for Low-rank Matrix Approximation, 
{\em SIAM J. Matrix Anal. Appl.}, {\bf 38, ~4}, 1454--1485, 2017.
Also see arXiv:1609.00048 January 2018. 


\bibitem{TYUC19}
J. A. Tropp, A. Yurtsever, M. Udell, V. Cevher, 
Streaming Low-Rank Matrix Approximation with an Application to Scientific Simulation,
{\em SIAM J. on Scientific Computing}, {\bf 41}, pp. A2430--A2463, 2019. Also
arXiv:1902.08651, submitted Feb 2019.

\bibitem{UT19}
M. Udell, A. Townsend, Why are big data matrices approximately of low rank?, {\em SIAM J. Math. Data Sci.}, {\bf 1}, 144-160, 2019.


\bibitem{X24}
J. Xia, Making the Nystr\"om method highly accurate for low-rank approximations, {\em SIAM J. Sci.
Comput.}, {\bf 46}, A1076--A1101  (2024)
\\  https://doi.org/10.1137/23M1585039.
Also arXiv:2307.05785 July 2023.


\bibitem{XXG12} 
J. Xia, Y. Xi, M. Gu, A superfast structured solver for Toeplitz linear systems via randomized sampling,
{\em SIAM J. Matrix Anal. Appl.}, {\bf 33},  837--858  (2012).  
 

\bibitem{ZO18}
 N. L. Zamarashkin, A.I. Osinsky,
 On the Existence of a Nearly Optimal Skeleton Approximation of a Matrix in the Frobenius Norm, {\em Doklady Mathematics}, {\bf 97,~2}, 164--166,  2018. 
 

\end{thebibliography}
\end{document}